\documentclass[11pt,a4paper,reqno]{amsart}

\usepackage[T1]{fontenc}
\usepackage[english]{babel}
\usepackage{mathtools}
\usepackage{amssymb}
\usepackage{libertinus}
\usepackage[cal=boondoxo,bb=ams]{mathalfa}
\usepackage{microtype}
\usepackage{enumitem}
\usepackage{array,booktabs,tabularx}
\usepackage{etoolbox}

\allowdisplaybreaks[2]
\usepackage[
 a4paper,
 left=30mm,
 right=30mm,
 top=27mm,
 bottom=30mm,
 headsep=8mm,
 footskip=13mm,
 heightrounded
]{geometry}

\setlist[itemize]{leftmargin=2.1em,itemsep=0.25em,topsep=0.45em,parsep=0pt}
\setlist[enumerate]{leftmargin=2.1em,itemsep=0.25em,topsep=0.45em,parsep=0pt}

\numberwithin{equation}{section}

\makeatletter
\def\@settitle{%
 \begin{center}
  \vspace*{-0.4em}
  {\normalfont\bfseries\fontsize{17}{21}\selectfont\@title\par}
  \vspace{0.7em}
 \end{center}}
\renewcommand\section{\@startsection{section}{1}{\z@}%
 {1.5\baselineskip plus 0.2\baselineskip minus 0.1\baselineskip}%
 {0.65\baselineskip}{\normalfont\Large\bfseries}}
\renewcommand\subsection{\@startsection{subsection}{2}{\z@}%
 {1.15\baselineskip plus 0.2\baselineskip minus 0.1\baselineskip}%
 {0.45\baselineskip}{\normalfont\large\bfseries}}
\renewcommand\subsubsection{\@startsection{subsubsection}{3}{\z@}%
 {0.9\baselineskip plus 0.15\baselineskip minus 0.1\baselineskip}%
 {0.35\baselineskip}{\normalfont\normalsize\bfseries}}
\makeatother

\makeatletter
\patchcmd{\@setauthors}{\centering\footnotesize}{\centering\normalsize}{}{}
\patchcmd{\@setauthors}{\MakeUppercase{\authors}}{\authors}{}{}
\patchcmd{\maketitle}{\uppercasenonmath\shorttitle}{}{}{}
\patchcmd{\maketitle}{\@nx\MakeUppercase{\the\toks@}}{\the\toks@}{}{}
\makeatother

\usepackage{hyperref}
\hypersetup{hidelinks,pdfencoding=auto}
\usepackage[nameinlink,capitalise,noabbrev]{cleveref}

\theoremstyle{plain}
\newtheorem{theorem}{Theorem}[section]
\newtheorem{lemma}[theorem]{Lemma}
\newtheorem{prop}[theorem]{Proposition}
\newtheorem{coro}[theorem]{Corollary}

\theoremstyle{definition}

\theoremstyle{remark}
\newtheorem{remark}[theorem]{Remark}

\newcommand{\ml}{\mathcal}
\newcommand{\mb}{\mathbb}
\newcommand{\dd}{\,\mathrm{d}}
\newcommand{\lin}{\mathrm{lin}}
\newcommand{\reg}{\mathrm{reg}}
\newcommand{\low}{\mathrm{low}}
\newcommand{\midf}{\mathrm{mid}}
\newcommand{\high}{\mathrm{high}}

\DeclareMathOperator{\supp}{supp}

\title[Critical exponent for strongly damped waves]{The critical exponent for the three-dimensional semilinear wave equation with strong damping}

\author[W. Chen]{Wenhui Chen}
\address{School of Mathematics and Information Science, Guangzhou University, Guangzhou 510006, P. R. China}
\email{wenhui.chen.math@gmail.com}

\date{}

\begin{document}

\begin{abstract}
	In this manuscript, we determine the critical exponent for the three-dimensional semilinear wave equation with strong damping and thereby resolve an open problem posed in 2014. The threshold for the power nonlinearity $|u|^p$ is
	\begin{align*}
	p=p_{\mathrm{crit}}=\frac{7}{3}.
	\end{align*}
	Sufficiently small data generate global in-time solutions for $p>\frac{7}{3}$, whereas there exist arbitrarily small smooth compactly supported data whose solutions blow up in finite time for $1<p\leqslant\frac{7}{3}$. Positivity of the full velocity fundamental solution, combined with a dimension-descent formula, yields a positive half-line kernel and replaces the finite propagation property unavailable for strongly damped waves. This leads to a nonlinear lower-bound parabolic iteration without radial symmetry or pointwise sign assumptions on the initial data. At the critical power, a refined slicing argument on moving shells converts the borderline logarithmic gain into the growth required for blow-up.
\end{abstract}

\maketitle

\section{Introduction}

The purpose of this paper is to determine the critical exponent for the three-dimensional Cauchy problem
\begin{align}\label{Main-Problem}
	\begin{cases}
		u_{tt}-\Delta u-\Delta u_t=|u|^p,&x\in\mb{R}^3,\ t>0,\\
		(u,u_t)(0,x)=(u_0,u_1)(x),&x\in\mb{R}^3,
	\end{cases}
\end{align}
with the power exponent $p>1$. The first systematic study, to the best of the author's knowledge, of \eqref{Main-Problem} established small data global in-time existence for $p>\frac{5}{2}$ and finite-time blow-up for $1<p\leqslant2$ in \cite[Theorem~2 and Theorem~9]{DAbbicco-Reissig=2014}, leaving the unresolved range
\begin{align*}
	2<p\leqslant\frac{5}{2}.
\end{align*}
 The determination of the critical exponent was explicitly left open in \cite[Remark~10]{DAbbicco-Reissig=2014} and was subsequently reiterated as an open problem in \cite{Kainane-Kainane-Reissig=2020,Jleli-Samet-Vetro=2021,Fino-Hamza=2022,Kirane-Fino-Kerbal-Laadhari=2024,DAbbicco-Lagioia=2025,Chen-Girardi=2025}. Additionally, the critical exponent for strongly non-effective $\sigma$-evolution equations was determined in \cite{DAbbicco-Ebert=2022} under the restriction $\sigma\neq1$. The wave endpoint $\sigma=1$, which includes \eqref{Main-Problem}, remained outside that theory. In the present paper, we close this gap in three dimensions and prove that the critical exponent is $p_{\mathrm{crit}}=\frac{7}{3}$.

The difficulty in closing this gap is intrinsic to the mixed hyperbolic-parabolic structure of the operator. The strong damping term $-\Delta u_t$ destroys finite propagation speed of the classical wave operator, while its dissipative effect varies substantially across the frequency space. At low frequencies, wave oscillations are modulated by diffusive damping, which gives rise to the diffusion-wave behavior established in \cite{Ikehata-Todorova-Yordanov=2013,Ikehata=2014,Ikehata-Onodera=2017}. At high frequencies, the equation enters an overdamping regime, where oscillations disappear and two distinct decay rates arise. The corresponding estimates for the linear problem require a different analysis from that in the oscillatory region, as shown in \cite{DAbbicco-Ebert=2025}. Consequently, the blow-up problem admits neither the support localization underlying the classical wave approach nor a reduction to a purely parabolic model. The available test function arguments do not detect the full contribution of the strong damping term and yield only the previously known blow-up range $1<p\leqslant 2$, as shown in \cite{DAbbicco-Reissig=2014,Fino-Hamza=2022}. Although the global side can be closed by combining frequency-localized estimates for the linear problem, the blow-up side requires a new lower-bound mechanism that incorporates wave propagation, diffusive spreading, and the growth generated by the power nonlinearity.

Several aspects of the linear problem associated with \eqref{Main-Problem} have been studied extensively. Optimal growth/decay estimates, diffusion phenomena, and large-time asymptotic profiles were obtained in \cite{Shibata=2000,DAbbicco-Ebert=2014,Ikehata-Todorova-Yordanov=2013,Ikehata=2014,Ikehata-Onodera=2017,Michihisa=2021,Ikehata-Takeda=2026}. More general semilinear problems for structurally damped wave and $\sigma$-evolution equations can be found in \cite{DAbbicco-Ebert=2017,Pham-Kainane-Reissig=2015,Dao-Reissig=2019,Ikehata-Takeda=2017,Kainane-Kainane-Reissig=2020} and the references therein. Related nonexistence results for time-dependent or additional damping terms were established in \cite{Jleli-Samet-Vetro=2021,Fino-Hamza=2022,Kirane-Fino-Kerbal-Laadhari=2024}. Nevertheless, none of these results provides matching global in-time existence and blow-up thresholds for the particular model considered here \eqref{Main-Problem}.

The value of the critical exponent is already suggested by the sharp decay of the low-frequency part of the homogeneous propagator. Let $S_1(t)$ denote the velocity (second) propagator for the homogeneous equation associated with \eqref{Main-Problem}, and let $S_1^{\low}(t)$ be its low-frequency component. It follows from \cite[Theorem~6, Remark~2 and Section~2.1]{DAbbicco-Ebert=2025} that
\begin{align*}
	\|S_1^{\low}(t)f(\cdot)\|_{L^q}\lesssim\langle t\rangle^{-\frac32+\frac{5}{2q}}\|f\|_{L^1}\ \ \text{for}\ \ q>2.
\end{align*}
If a small nonlinear solution $u=u(t,x)$ obeys the corresponding decay with size $\varepsilon$ for its initial data, namely,
\begin{align*}
	\|u(t,\cdot)\|_{L^p}\lesssim\varepsilon\langle t\rangle^{-\frac32+\frac{5}{2p}}\ \ \text{for some}\ \ p>2,
\end{align*}
then
\begin{align*}
	\||u(t,\cdot)|^p\|_{L^1}=\|u(t,\cdot)\|_{L^p}^p\lesssim\varepsilon^p\langle t\rangle^{-\frac{3p-5}{2}}.
\end{align*}
The Duhamel source is therefore expected to be integrable in time precisely when
\begin{align*}
	-\frac{3p-5}{2}<-1\ \ \Leftrightarrow\ \ p>\frac{7}{3}.
\end{align*}
This identifies $\frac{7}{3}$ as the natural candidate threshold for the global in-time existence for \eqref{Main-Problem}. Establishing the matching blow-up result is substantially more delicate because the strongly damped wave equation lacks finite propagation speed, while at the critical exponent the subcritical iteration only reproduces the first lower bound and yields no logarithmic growth.

We demonstrate that $\frac{7}{3}$ is the critical exponent for \eqref{Main-Problem}. For the supercritical range $p>\frac{7}{3}$, small data global in-time existence is obtained through a frequency-separated fixed point argument that combines the sharp low-frequency decay with suitable $L^2$-based estimates for the remaining frequencies. Together with the previously known blow-up result for $1<p\leqslant2$, the new blow-up analysis for
\begin{align*}
	2<p\leqslant\frac{7}{3}
\end{align*}
completes the classification.

To identify a new blow-up mechanism for the strongly damped waves, our main contribution is to construct a substitute for the finite propagation property lost under strong damping. We prove positivity of the full velocity fundamental solution and use a dimension-descent identity to derive a positive Dirichlet kernel on the half-line. Spherical averaging and Jensen's inequality may reduce the three-dimensional problem to a nonlinear lower-bound integral inequality on the half-line. This construction replaces the support localization available for semilinear classical wave equations. It requires neither radial symmetry nor pointwise positivity of the initial data, and the only sign condition is imposed on the spatial integral of the initial velocity $u_1$. In the subcritical range $2<p<\frac{7}{3}$, a finite iteration in a parabolic characteristic region produces a nondecaying lower bound. This lower bound is then plugged into the Duhamel term to obtain polynomial growth in a time-dependent strict interior cone, after which comparison with a scalar ODE forces finite-time blow-up. At the critical exponent, the power iteration used in the subcritical case only reproduces the existing lower bound and yields no further growth. Alternately, a refined iteration on moving shells turns the first logarithmic gain into a super-polynomial lower bound, which is then sufficient for a similar ODE blow-up argument.

\medskip
\paragraph{Notation.}
Throughout this paper, the constants $C$ and $c$ are positive and may change from line to line. We write $f\lesssim g$ if $f\leqslant Cg$, $f\gtrsim g$ if $g\lesssim f$, and $f\approx g$ if both $f\lesssim g$ and $g\lesssim f$ hold. A subscript attached to these symbols indicates the parameters on which the implicit constant is allowed to depend. The underlying spatial domain is always $\mb R^3$ whenever no domain is displayed explicitly.  The symbol $\ast_{(x)}$ denotes convolution with respect to the spatial variable. The Fourier transform and its inverse are defined, respectively, via
\begin{align*}
	\widehat f(\zeta):=\ml F[f](\zeta)
	&:=\int_{\mb R^n}\mathrm{e}^{-ix\cdot\zeta}f(x)\dd x,\\
	\ml F^{-1}[g](x)
	&:=\frac{1}{(2\pi)^n}\int_{\mb R^n}\mathrm{e}^{ix\cdot\zeta}g(\zeta)\dd\zeta.
\end{align*}
 For a function $h=h(t)$, its Laplace transform is written as
\begin{align*}
	\widetilde h(s):=\ml L[h](s):=\int_0^\infty\mathrm{e}^{-st}h(t)\dd t.
\end{align*}
We define $\langle t\rangle:=1+t$ for $t\geqslant0$ and $\langle z\rangle:=(1+|z|^2)^{\frac{1}{2}}$ for $z\in\mb R^n$. Moreover, $B_R:=\{x\in\mb R^3:|x|<R\}$. For a function $f=f(x)$ on $\mb R^3$, its spherical mean is defined by
\begin{align*}
	\ml M[f](r):=\frac{1}{4\pi}\int_{\mb S^2}f(r\omega)\dd\omega\ \ \text{for}\ \ r\geqslant0,
\end{align*}
where $\mathbb S^2$ is the unit sphere in $\mathbb R^3$ equipped with the standard surface measure $\mathrm{d}\omega$.

\section{Main results}\label{Section-Main}

Let us consider the same initial data space as \cite[Theorem 2]{DAbbicco-Reissig=2014}, namely,
\begin{align*}
	\ml{A}:=(L^1\cap H^2)\times(L^1\cap L^2),
\end{align*}
endowed with its norm
\begin{align*}
	\|(u_0,u_1)\|_{\ml A}:=\|u_0\|_{L^1}+\|u_0\|_{H^2}+\|u_1\|_{L^1}+\|u_1\|_{L^2}.
\end{align*}

Let $S_0(t)$ and $S_1(t)$ denote the position $u|_{t=0}$ and velocity $u_t|_{t=0}$ propagators for the homogeneous equation associated with \eqref{Main-Problem}, respectively. Thus, the solution to the homogeneous problem with initial data $(v_0,v_1)$ is given by $S_0(t)v_0+S_1(t)v_1$.

For $T\in(0,\infty]$, a function
\begin{align*}
	u\in\ml{C}\bigl([0,T),H^2\bigr)\cap\ml{C}^1\bigl([0,T),L^2\bigr)
\end{align*}
is called a mild Sobolev (energy) solution to \eqref{Main-Problem} on $[0,T)$ if
\begin{align}\label{Mild-Formula}
	u(t,\cdot)=S_0(t)u_0(\cdot)+S_1(t)u_1(\cdot)+\int_0^tS_1(t-s)|u(s,\cdot)|^p\dd s\ \ \text{in}\ \ H^2
\end{align}
holds for every $t\in[0,T)$. A mild Sobolev solution is called maximal if it cannot be extended to a larger time interval in the same class, whose maximal existence time is denoted by $T_{\max}$.

For $p>\frac{7}{3}$, let us introduce two parameters $\alpha_p:=\frac{3}{2}-\frac{5}{2p}$ and $\beta_p:=\min\left\{\frac{5}{4},\gamma_p\right\}$ carrying $\gamma_p:=p\alpha_p=\frac{3p-5}{2}$ in the forthcoming decay rates.

\begin{theorem}[Small data global in-time existence]\label{Thm-Global}
	Let $p>p_{\mathrm{crit}}=\frac{7}{3}$. There exists $\delta=\delta(p)>0$ such that, for every $(u_0,u_1)\in\ml A$ with $\|(u_0,u_1)\|_{\ml A}\leqslant\delta$, the Cauchy problem \eqref{Main-Problem} admits a unique global in-time mild Sobolev solution. Furthermore, the solution satisfies the following decay estimates:
	\begin{align*}
		\|u(t,\cdot)\|_{L^p}&\lesssim\langle t\rangle^{-\alpha_p}\|(u_0,u_1)\|_{\ml A},\\
		\|u(t,\cdot)\|_{L^2}&\lesssim\langle t\rangle^{-\frac14}\|(u_0,u_1)\|_{\ml A},\\
		\|u_t(t,\cdot)\|_{L^2}+\|\nabla u(t,\cdot)\|_{L^2}&\lesssim\langle t\rangle^{-\frac34}\|(u_0,u_1)\|_{\ml A},\\
		\|\nabla^2u(t,\cdot)\|_{L^2}&\lesssim\langle t\rangle^{-\beta_p}\|(u_0,u_1)\|_{\ml A},
	\end{align*}
	for every $t\geqslant0$.
\end{theorem}

\begin{remark}\label{Remark-No-Decay-Loss}
	Compared with the decay estimates for the corresponding linear problem, the nonlinear solution exhibits no loss of decay in the $L^p$, $L^2$, and energy norms for every $p>\frac{7}{3}$. The second-order $L^2$-estimate has the same decay rate as the linear solution precisely when $\gamma_p=\frac{3p-5}{2}\geqslant\frac{5}{4}$, that is, when $p\geqslant\frac{5}{2}$.
\end{remark}

\begin{theorem}[Finite-time blow-up]\label{Thm-Blowup}
	Let $2<p\leqslant p_{\mathrm{crit}}=\frac{7}{3}$. Assume that $u_0,u_1\in\ml C_0^\infty$ satisfy
	\begin{align}\label{Positive-Mean}
		M_1:=\int_{\mb{R}^3}u_1(x)\dd x>0.
	\end{align}
	Then the maximal mild Sobolev solution to \eqref{Main-Problem} satisfies $T_{\max}<\infty$.
\end{theorem}

\begin{remark}
	The condition $M_1>0$ is closely related to the averaged sign conditions used in test-function arguments for damped and strongly damped wave equations, which can be formulated for integrable initial data without compact support assumptions (see, for example, \cite[Theorem~1]{DAbbicco-Lucente=2013} and \cite[Theorem~1]{Fino-Hamza=2022}). In the present proof, this averaged positivity is converted into a pointwise lower bound on a moving shell, which initiates the nonlinear iteration without radial symmetry or pointwise sign assumptions on the initial data.
\end{remark}

\begin{remark}\label{Remark-Lower-Range}
	For $1<p\leqslant2$, finite-time blow-up for suitable nonnegative compactly supported data was established in \cite[Theorem~9]{DAbbicco-Reissig=2014}. The assumptions of that result are preserved under positive rescaling, so finite-time blow-up also occurs for arbitrarily small smooth compactly supported data in this range. Theorem~\ref{Thm-Blowup} closes the remaining interval $2<p\leqslant\frac{7}{3}$. 
\end{remark}

The preceding results yield the exact threshold in three dimensions.

\begin{coro}[Critical exponent]\label{Coro-Critical}
	The critical exponent for \eqref{Main-Problem} is
	\begin{align*}
	p_{\mathrm{crit}}=\frac{7}{3}.
	\end{align*}
	More precisely, all sufficiently small data in $\ml A$ generate global in-time solutions when $p>p_{\mathrm{crit}}$, whereas for every $1<p\leqslant p_{\mathrm{crit}}$ there exist smooth compactly supported data with arbitrarily small $\ml A$-norm whose corresponding solutions blow up in finite time.
\end{coro}

\section{Small data global in-time existence}\label{Section-Global}

\subsection{Linear preliminaries}
We collect the properties of the homogeneous propagators that will be used in the nonlinear analysis. The sharp low-frequency $L^1-L^q$ estimate is a direct specialization of \cite[Theorem~6, Remark~2 and Section~2.1]{DAbbicco-Ebert=2025}. The remaining bounds are standard consequences of the explicit Fourier multipliers. The symbol estimates near the double root and in the high-frequency region are deferred to \cref{Appendix-Symbols}.
\subsubsection{Fourier representation and frequency decomposition}\label{Section-Linear}

Taking the partial Fourier transform with respect to $x$ of the homogeneous equation associated with \eqref{Main-Problem}, we obtain
\begin{align}\label{Linear-Fourier-ODE}
	\widehat v_{tt}+|\zeta|^2\widehat v_t+|\zeta|^2\widehat v=0.
\end{align}
For notational convenience, let us take $\rho:=|\zeta|$.
\begin{itemize}
	\item For $0\leqslant\rho<2$, let us define $\omega(\rho):=\rho\sqrt{1-\frac{\rho^2}{4}}$. The position and velocity multipliers are given by
	\begin{align*}
		\widehat S_0(t,\zeta)&=\mathrm{e}^{-\frac{t\rho^2}{2}}\left(\cos\bigl(t\omega(\rho)\bigr)+\frac{\rho^2}{2\omega(\rho)}\sin\bigl(t\omega(\rho)\bigr)\right),\\
		\widehat S_1(t,\zeta)&=\mathrm{e}^{-\frac{t\rho^2}{2}}\frac{\sin\bigl(t\omega(\rho)\bigr)}{\omega(\rho)}.
	\end{align*}
	\item For $\rho>2$, the characteristic roots are $\lambda_\pm(\rho):=\frac{-\rho^2\pm\rho\sqrt{\rho^2-4}}{2}$ so that
	\begin{align*}
		\widehat S_0(t,\zeta)&=\frac{-\lambda_-(\rho)\mathrm{e}^{\lambda_+(\rho)t}+\lambda_+(\rho)\mathrm{e}^{\lambda_-(\rho)t}}{\lambda_+(\rho)-\lambda_-(\rho)},\\
		\widehat S_1(t,\zeta)&=\frac{\mathrm{e}^{\lambda_+(\rho)t}-\mathrm{e}^{\lambda_-(\rho)t}}{\lambda_+(\rho)-\lambda_-(\rho)}.
	\end{align*}
	\item For $\rho=2$, the apparent singularity at $|\zeta|=2$ is removable so that
	\begin{align*}
		\widehat S_0(t,\zeta)&=(1+2t)\mathrm{e}^{-2t},\\
		\widehat S_1(t,\zeta)&=t\mathrm{e}^{-2t}.
	\end{align*}	
\end{itemize}

We fix smooth radial cutoff functions
\begin{align*}
	\chi_{\low},\chi_{\midf},\chi_{\high}\in \ml{C}^\infty\bigl([0,\infty)\bigr)\ \ \text{such that}\ \ \chi_{\low}+\chi_{\midf}+\chi_{\high}=1.
\end{align*} We assume that $\chi_{\low}$ is supported in a sufficiently small neighborhood of the origin and equals $1$ near $\rho=0$. We also assume that $\chi_{\high}$ is supported in $[R_0,\infty)$ and equals $1$ for all sufficiently large $\rho$, where $R_0>4$ without loss of generality. For all $\nu\in\{\low,\midf,\high\}$ and $j\in\{0,1\}$, we define the frequency-localized propagators by
\begin{align*}
	S_j^\nu(t)f(x):=\mathcal F^{-1}_{\zeta\to x}\left(\chi_\nu(|\zeta|)\widehat S_j(t,\zeta)\widehat f(\zeta)\right).
\end{align*}

\subsubsection{Estimates used in the nonlinear analysis}

For $j\in\{0,1\}$, we write
\begin{align*}
	S_j^{\midf+\high}(t):=S_j^{\midf}(t)+S_j^{\high}(t).
\end{align*}

\begin{prop}[Estimates for the linear problem]\label{Prop-Full-Linear}
	Let $q\in(2,\infty)$, $f\in L^1\cap L^2$ and $g\in L^1\cap H^2$. There exists $c>0$ such that, for every $t\geqslant0$,
	\begin{align}\label{Full-Lq-Compact}
		\|S_1(t)f(\cdot)\|_{L^q}+\|S_0(t)g(\cdot)\|_{L^q}\lesssim\langle t\rangle^{-\frac32+\frac{5}{2q}}\bigl(\|f\|_{L^1}+\|g\|_{L^1}\bigr)+\mathrm{e}^{-ct}\bigl(\|f\|_{L^2}+\|g\|_{H^2}\bigr).
	\end{align}
	One also has
	\begin{align*}
		\sum_{k=0}^2\langle t\rangle^{\frac{2k+1}{4}}\|\nabla^kS_1(t)f(\cdot)\|_{L^2}+\sum_{k=0}^2\langle t\rangle^{\frac{2k+3}{4}}\|\nabla^kS_0(t)g(\cdot)\|_{L^2}\lesssim\|f\|_{L^1\cap L^2}+\|g\|_{L^1\cap H^2},
	\end{align*}
	and
	\begin{align}\label{Full-Time-Compact}
		\langle t\rangle^{\frac{3}{4}}\|\partial_tS_1(t)f(\cdot)\|_{L^2}+\langle t\rangle^{\frac{5}{4}}\|\partial_tS_0(t)g(\cdot)\|_{L^2}\lesssim\|f\|_{L^1\cap L^2}+\|g\|_{L^1\cap H^2}.
	\end{align}
	Particularly, the middle- and high-frequency parts fulfill
	\begin{align*}
		\|S_1^{\midf+\high}(t)f(\cdot)\|_{H^2}+\|\partial_tS_1^{\midf+\high}(t)f(\cdot)\|_{L^2}&\lesssim\mathrm{e}^{-ct}\|f\|_{L^2},\\[-1mm]
		\|S_0^{\midf+\high}(t)g(\cdot)\|_{H^2}+\|\partial_tS_0^{\midf+\high}(t)g(\cdot)\|_{L^2}&\lesssim\mathrm{e}^{-ct}\|g\|_{H^2}.
	\end{align*}
\end{prop}

\begin{proof}
	In the notation of \cite[Theorem~6]{DAbbicco-Ebert=2025}, we take dimension $n=3$, input exponent equal to $1$ and output exponent $q>2$. The quantity governing the low-frequency loss (still by their notation) is
	\begin{align*}
		d(1,q)=3\left(1-\frac1q\right)+2\left(\frac1q-\frac12\right)=2-\frac{1}{q}>1.
	\end{align*}
	Hence, the decay exponent in \cite[Theorem~6 and Section~2.1]{DAbbicco-Ebert=2025} becomes
	\begin{align*}
		1-3\left(1-\frac1q\right)+\frac{d(1,q)-1}{2}=-\frac32+\frac{5}{2q},
	\end{align*}
	which justifies the low-frequency velocity estimate in \eqref{Full-Lq-Compact}. Since the low-frequency position multiplier is uniformly bounded, its $L^1-L^q$ decay is stronger than the rate stated in \eqref{Full-Lq-Compact}.
	
	Thanks to $\omega(\rho)\approx\rho$ on the support of $\chi_{\low}$, i.e. $\supp\chi_{\low}\subset[0,2\rho_0]$ for $0<\rho_0<1$ sufficiently small, Plancherel's theorem indicates, for $k=0,1,2$, that
	\begin{align*}
		\|\nabla^kS_1^{\low}(t)f(\cdot)\|_{L^2}^2\lesssim\|f\|_{L^1}^2\int_0^{2\rho_0}\rho^{2k}\mathrm{e}^{-ct\rho^2}\sin^2\bigl(t\omega(\rho)\bigr)\dd\rho\lesssim\langle t\rangle^{-\frac{2k+1}{2}}\|f\|_{L^1}^2.
	\end{align*}
	The boundedness of the position multiplier yields
	\begin{align*}
		\|\nabla^kS_0^{\low}(t)g(\cdot)\|_{L^2}^2\lesssim\|g\|_{L^1}^2\int_0^{2\rho_0}\rho^{2k+2}\mathrm{e}^{-ct\rho^2}\dd\rho\lesssim\langle t\rangle^{-\frac{2k+3}{2}}\|g\|_{L^1}^2.
	\end{align*}
	Moreover, the low-frequency multiplier of $\partial_tS_1(t)$ is bounded by $C\mathrm{e}^{-ct\rho^2}$, while $\partial_t\widehat S_0(t,\zeta)=-\rho^2\widehat S_1(t,\zeta)$. These observations give the low-frequency estimates in \eqref{Full-Time-Compact}.
	
	The middle-frequency estimates follow from \eqref{Middle-Matrix}, whereas the high-frequency bounds follow from \eqref{ab-Asymptotics}, Lemma~\ref{Lemma-High-Small-Time} and Lemma~\ref{Lemma-High-Symbols}. The corresponding $L^q$-bounds follow from the embedding $H^2\hookrightarrow L^q$. Combining the three frequency regions completes the proof.
\end{proof}

The bounded-time estimates of Proposition~\ref{Prop-Full-Linear}, together with the local Lipschitz continuity of $u\mapsto |u|^p$ from $H^2$ into $L^1\cap L^2$, lead to the standard local in-time well-posedness result. In particular, for every $p>2$ and $(u_0,u_1)\in\mathcal A$, there exists a unique maximal mild Sobolev solution on an interval $[0,T_{\max})$, where $T_{\max}\in(0,\infty]$.

\subsection{Proof of Theorem~\ref{Thm-Global}}

For $T\in(0,\infty]$, let $X(T)$ be the space of all functions
\begin{align*}
	u\in\ml C\bigl([0,T),H^2\bigr)\cap\ml C^1\bigl([0,T),L^2\bigr)
\end{align*}
for which the norm
\begin{align}\label{X-Norm}
	\|u\|_{X(T)}:=\sup_{0\leqslant t<T}\Big(&\langle t\rangle^{\alpha_p}\|u(t,\cdot)\|_{L^p}+\langle t\rangle^{\frac{1}{4}}\|u(t,\cdot)\|_{L^2}\notag\\
	&+\langle t\rangle^{\frac{3}{4}}\bigl(\|\nabla u(t,\cdot)\|_{L^2}+\|u_t(t,\cdot)\|_{L^2}\bigr)+\langle t\rangle^{\beta_p}\|\nabla^2u(t,\cdot)\|_{L^2}\Big)
\end{align}
is finite. Equipped with this norm, $X(T)$ is a Banach space.

We first establish the estimates for the nonlinear source, whose decay exponent is denoted by $\eta_p:=\frac{1}{4}+(p-1)\delta_p$ carrying  $\delta_p:=\frac{1}{16}+\frac{3\beta_p}{4}$.

\begin{lemma}[Nonlinear source bounds]\label{Lemma-Nonlinear-Bounds}
	Let $u\in X(T)$. Then
	\begin{align}
		\||u(t,\cdot)|^p\|_{L^1}&\lesssim\langle t\rangle^{-\gamma_p}\|u\|_{X(T)}^p,\notag\\
		\||u(t,\cdot)|^p\|_{L^2}&\lesssim\langle t\rangle^{-\eta_p}\|u\|_{X(T)}^p,\label{Nonlinear-L2}
	\end{align}
	for every $t\in[0,T)$. The exponents defined above also satisfy
	\begin{align}\label{Eta-Greater-Beta}
		\eta_p>\max\left\{\alpha_p,\frac{3}{4},\beta_p\right\}
	\end{align}
	whenever $p>\frac{7}{3}$.
\end{lemma}

\begin{proof}
	The first estimate follows from
	\begin{align*}
		\||u(t,\cdot)|^p\|_{L^1}\lesssim\langle t\rangle^{-p\alpha_p}\|u\|_{X(T)}^p=\langle t\rangle^{-\gamma_p}\|u\|_{X(T)}^p.
	\end{align*}
	For the second estimate, the three-dimensional Gagliardo-Nirenberg inequality gives
	\begin{align}\label{GN-Infinity}
		\|u(t,\cdot)\|_{L^\infty}\lesssim\|u(t,\cdot)\|_{L^2}^{\frac{1}{4}}\|\nabla^2u(t,\cdot)\|_{L^2}^{\frac{3}{4}}\lesssim\langle t\rangle^{-\delta_p}\|u\|_{X(T)}.
	\end{align}
	It follows from \eqref{GN-Infinity} and the definition of the $X(T)$-norm that
	\begin{align*}
		\||u(t,\cdot)|^p\|_{L^2}\leqslant\|u(t,\cdot)\|_{L^\infty}^{p-1}\|u(t,\cdot)\|_{L^2}\lesssim\langle t\rangle^{-(p-1)\delta_p-\frac{1}{4}}\|u\|_{X(T)}^p,
	\end{align*}
	which proves \eqref{Nonlinear-L2}. The strict inequality \eqref{Eta-Greater-Beta} is valid because
	\begin{itemize}
		\item when $\frac{7}{3}<p<\frac{5}{2}$, one has
			\begin{align*}
			16(\eta_p-\beta_p)=18p^2-71p+73>0\ \ \text{and}\ \ \beta_p>1>\max\left\{\alpha_p,\frac{3}{4}\right\};
		\end{align*}
		\item when $p\geqslant\frac{5}{2}$, one has
		\begin{align*}
		\eta_p-\beta_p=p-2>0\ \ \text{and}\ \ \eta_p-\alpha_p=p-\frac{9}{4}+\frac{5}{2p}>0.
		\end{align*}
	\end{itemize}
Our proof is complete.
\end{proof}

We shall repeatedly apply the following elementary convolution estimate (see, for example, \cite[Lemma~2.4~(i)]{Ikehata-Takeda=2017}):
\begin{align}\label{Time-Convolution}
	\int_0^t\langle t-s\rangle^{-a}\langle s\rangle^{-b}\dd s\lesssim\langle t\rangle^{-\min\{a,b\}}\ \ \text{for}\ \ a>0,\ b>1.
\end{align}

Strongly motivated by the mild formulation \eqref{Mild-Formula}, we define the nonlinear map
\begin{align*}
	\Phi[u](t,x):=S_0(t)u_0(x)+S_1(t)u_1(x)+\int_0^tS_1(t-s)|u(s,x)|^p\dd s
\end{align*}
from $u\in X(T)$ for any $T\in(0,\infty]$.

\begin{prop}[Mapping estimate]\label{Prop-Mapping}
	Let $p>\frac{7}{3}$ and $T\in(0,\infty]$. For every $u\in X(T)$, one has $\Phi[u]\in X(T)$ satisfying
	\begin{align}\label{Mapping-Estimate}
		\|\Phi[u]\|_{X(T)}\lesssim \|(u_0,u_1)\|_{\ml A}+\|u\|_{X(T)}^p.
	\end{align}
\end{prop}

\begin{proof}
	The homogeneous (linear) part is controlled by Proposition~\ref{Prop-Full-Linear}. We now focus on the Duhamel (nonlinear) term. For its low-frequency $L^p$ component, Proposition~\ref{Prop-Full-Linear}, Lemma~\ref{Lemma-Nonlinear-Bounds} and \eqref{Time-Convolution} give
	\begin{align*}
		\left\|\int_0^tS_1^{\low}(t-s)|u(s,\cdot)|^p\dd s\right\|_{L^p}&\lesssim\|u\|_{X(T)}^p\int_0^t\langle t-s\rangle^{-\alpha_p}\langle s\rangle^{-\gamma_p}\dd s\lesssim\langle t\rangle^{-\alpha_p}\|u\|_{X(T)}^p,
	\end{align*}
	where we used $\gamma_p>1$ and $\gamma_p=p\alpha_p>\alpha_p$. The low-frequency $L^2$-estimates in Proposition~\ref{Prop-Full-Linear} analogously yield
	\begin{align*}
		\left\|\nabla^k\int_0^tS_1^{\low}(t-s)|u(s,\cdot)|^p\dd s\right\|_{L^2}\lesssim\langle t\rangle^{-\min\left\{\frac{2k+1}{4},\gamma_p\right\}}\|u\|_{X(T)}^p
	\end{align*}
	for all $k\in\{0,1,2\}$. Recalling $\gamma_p>1$, the previous minimum equals $\frac{2k+1}{4}$ for $k\in\{0,1\}$, while for $k=2$ it equals $\min\left\{\frac{5}{4},\gamma_p\right\}=\beta_p$. The corresponding estimate for the time derivative has its decay rate $\langle t\rangle^{-\frac{3}{4}}$.
	
	For the middle- and high-frequency parts, Proposition~\ref{Prop-Full-Linear} and \eqref{Nonlinear-L2} imply
	\begin{align*}
		&\left\|\int_0^tS_1^{\midf+\high}(t-s)|u(s,\cdot)|^p\dd s\right\|_{H^2}+\left\|\partial_t\int_0^tS_1^{\midf+\high}(t-s)|u(s,\cdot)|^p\dd s\right\|_{L^2}\\
		&\lesssim\|u\|_{X(T)}^p\int_0^t\mathrm{e}^{-c(t-s)}\langle s\rangle^{-\eta_p}\dd s\\
		&\lesssim\langle t\rangle^{-\eta_p}\|u\|_{X(T)}^p.
	\end{align*}
	The last estimate follows by splitting the time integral at $\frac{t}{2}$. In view of \eqref{Eta-Greater-Beta}, this decay is faster than $\langle t\rangle^{-\frac{1}{4}}$, $\langle t\rangle^{-\frac{3}{4}}$ and $\langle t\rangle^{-\beta_p}$. Moreover, the embedding $H^2\hookrightarrow L^p$ holds for every finite $p\geqslant2$. Since $\eta_p>\alpha_p$, we also obtain
	\begin{align*}
		\left\|\int_0^tS_1^{\midf+\high}(t-s)|u(s,\cdot)|^p\dd s\right\|_{L^p}\lesssim\langle t\rangle^{-\alpha_p}\|u\|_{X(T)}^p.
	\end{align*}
	All components of the norm \eqref{X-Norm} are therefore controlled, which proves \eqref{Mapping-Estimate}.
\end{proof}

\begin{prop}[Difference estimate]\label{Prop-Difference}
	Let $p>\frac{7}{3}$ and $T\in(0,\infty]$. For every $u,\tilde u\in X(T)$, one has
	\begin{align}\label{Difference-Estimate}
		\|\Phi[u]-\Phi[\tilde u]\|_{X(T)}\lesssim\left(\|u\|_{X(T)}^{p-1}+\|\tilde u\|_{X(T)}^{p-1}\right)\|u-\tilde u\|_{X(T)}.
	\end{align}
\end{prop}

\begin{proof}
	The pointwise difference inequality and the decay contained in the $X(T)$-norm give
	\begin{align*}
		\||u(t,\cdot)|^p-|\tilde u(t,\cdot)|^p\|_{L^1}&\lesssim\langle t\rangle^{-\gamma_p}\left(\|u\|_{X(T)}^{p-1}+\|\tilde u\|_{X(T)}^{p-1}\right)\|u-\tilde u\|_{X(T)},\\
		\||u(t,\cdot)|^p-|\tilde u(t,\cdot)|^p\|_{L^2}&\lesssim\langle t\rangle^{-\eta_p}\left(\|u\|_{X(T)}^{p-1}+\|\tilde u\|_{X(T)}^{p-1}\right)\|u-\tilde u\|_{X(T)}.
	\end{align*}
	Applying the same low-frequency and middle- and high-frequency decomposition as in the proof of Proposition~\ref{Prop-Mapping} yields \eqref{Difference-Estimate} similarly.
\end{proof}

\begin{proof}[Proof of Theorem~\ref{Thm-Global}]
	Let $C>0$ be a common constant independent of $T$ in \eqref{Mapping-Estimate} and \eqref{Difference-Estimate}. We now set
	\begin{align*}
	R:=2C\|(u_0,u_1)\|_{\ml A}.
	\end{align*} We choose $\delta>0$ sufficiently small so that $C(2C\delta)^{p-1}\leqslant\frac{1}{4}$. If $\|(u_0,u_1)\|_{\ml A}\leqslant\delta$,  then $CR^{p-1}\leqslant\frac{1}{4}$. For every $u\in X(\infty)$ satisfying $\|u\|_{X(\infty)}\leqslant R$, Proposition~\ref{Prop-Mapping} immediately leads to
	\begin{align*}
		\|\Phi[u]\|_{X(\infty)}\leqslant C\|(u_0,u_1)\|_{\ml A}+CR^p<R.
	\end{align*}
	Moreover, if $u,\tilde u$ belong to the same closed ball, then Proposition~\ref{Prop-Difference} yields
	\begin{align*}
		\|\Phi[u]-\Phi[\tilde u]\|_{X(\infty)}&\leqslant 2CR^{p-1}\|u-\tilde u\|_{X(\infty)}\leqslant\frac{1}{2}\|u-\tilde u\|_{X(\infty)}.
	\end{align*}
	In other words, $\Phi$ has a unique fixed point in this ball. This fixed point is a unique global in-time mild Sobolev solution and satisfies $\|u\|_{X(\infty)}\lesssim\|(u_0,u_1)\|_{\ml A}$. The estimates in Theorem \ref{Thm-Global} now follow from the definition of the $X(\infty)$-norm.
\end{proof}

\section{Positivity and pointwise estimates for the fundamental kernels}\label{Section-Kernels}

\subsection{Laplace representation and positivity}

Let $G_1^{(3)}(t,x)$ denote the three-dimensional velocity fundamental solution (the superscript ``$(3)$'' indicates the space dimension $n=3$), initially understood as the tempered distribution defined by
\begin{align*}
	\widehat{G_1^{(3)}}(t,\zeta)=\widehat S_1(t,\zeta)\equiv \widehat S_1(t,|\zeta|).
\end{align*}
Equivalently,
\begin{align*}
S_1(t)f(x)=G_1^{(3)}(t,x)\ast_{(x)}f(x).
\end{align*} Since its Fourier transform depends only on $|\zeta|$, the kernel $G_1^{(3)}$ is radial in the spatial variable. For $r=|x|$, we write $G_1^{(3)}(t,r):=G_1^{(3)}(t,x)$.  While the Fourier representation is suited to the frequency analysis developed in \cref{Section-Linear}, the following representation of its time-Laplace transform is going to reveal the positivity structure of the kernel (motivated by \cite{Hanyga-Seredynska=2010} for a broad class of scalar viscoelastic equations). An application of the partial Laplace transform with respect to $t$ addresses
\begin{align*}
	\widetilde G_1^{(3)}(s,r)&:=\ml L_{t\to s}\bigl[G_1^{(3)}\bigr](s,r)\\
	&:=\int_0^\infty \mathrm{e}^{-st}G_1^{(3)}(t,r)\dd t \ \ \text{for}\  \ s>0.
\end{align*}
Before stating the lemma and its proof, we briefly recall that a function $f\in\mathcal C^\infty\bigl((0,\infty)\bigr)$ is completely monotone if
\begin{align*}
(-1)^k f^{(k)}(s)\geqslant0
\ \ \text{for all}\ \  k\in\mathbb N_0,\  s>0,	
\end{align*}
whereas a nonnegative function $\psi\in\mathcal C^\infty\bigl((0,\infty)\bigr)$ is called a Bernstein function if $\psi'$ is completely monotone. By Bernstein's theorem, completely monotone functions are precisely the Laplace transforms of nonnegative Radon measures (see \cite[Chapter 1 and Chapter 3]{Schilling-Song-Vondracek=2012}).

\begin{lemma}[Laplace representation]\label{Lemma-Laplace-Kernel}
	For $s>0$ and $r>0$, one has
	\begin{align}\label{Laplace-G3}
		\widetilde G_1^{(3)}(s,r)=\frac1{4\pi r(1+s)}\mathrm{e}^{-\frac{rs}{\sqrt{1+s}}}.
	\end{align}
	Furthermore, for every $r>0$, there exists a nonnegative Radon measure
	$\mu_r$ on $[0,\infty)$ satisfying
	\begin{align*}
		\widetilde G_1^{(3)}(s,r)=\int_{[0,\infty)}\mathrm{e}^{-st}\,\mu_r(\mathrm{d}t)\ \ \text{for}\ \  s>0.
	\end{align*}
\end{lemma}

\begin{proof}
	Taking the Laplace transform in time at the level of the velocity multiplier, we are able to deduce that
	\begin{align}\label{Supp-01}
		\widetilde{\widehat G}_1^{(3)}(s,\zeta)=\frac{1}{s^2+s|\zeta|^2+|\zeta|^2}=\frac{1}{1+s}\frac{1}{|\zeta|^2+\frac{s^2}{1+s}}.
	\end{align}
	For $\lambda>0$, the three-dimensional fundamental solution of the modified Helmholtz operator (see, for example, \cite{Aronszajn-Smith=1961}) satisfies
	\begin{align*}
		\mathcal{F}_{\zeta\to x}^{-1}\left[\frac{1}{|\zeta|^2+\lambda^2}\right](x)=\frac{1}{4\pi|x|}\mathrm{e}^{-\lambda|x|}\ \ \text{for}\ \  x\in\mathbb{R}^3\setminus\{0\}.
	\end{align*}
	Choosing $\lambda=\frac{s}{\sqrt{1+s}}$ and writing $r=|x|$, we obtain \eqref{Laplace-G3}.
	
	It remains to establish the complete monotonicity of the right-hand side of \eqref{Laplace-G3}. Actually, we know that
	\begin{align*}
		\psi(s):=\frac{s}{\sqrt{1+s}} \ \ \Rightarrow\ \ \psi'(s)=\frac{1}{2}(1+s)^{-\frac{1}{2}}+\frac{1}{2}(1+s)^{-\frac{3}{2}}>0.
	\end{align*}
	Hence, $\psi'$ is completely monotone, and consequently $\psi$ is a Bernstein function. It follows from \cite[Theorem~3.6]{Schilling-Song-Vondracek=2012} that $s\mapsto\mathrm{e}^{-r\psi(s)}$ is completely monotone for every fixed $r>0$. According to the well-known fact that $s\mapsto(1+s)^{-1}$ is also completely monotone and the product of two completely monotone functions remains completely monotone, the
	right-hand side of \eqref{Laplace-G3} is completely monotone. Bernstein's theorem \cite[Theorem~1.4]{Schilling-Song-Vondracek=2012} therefore yields a
	nonnegative Radon measure $\mu_r$ on $[0,\infty)$ whose Laplace transform is given by \eqref{Laplace-G3}.
\end{proof}

The last lemma yields a nonnegative measure $\mu_r$ in the time variable, nevertheless it does not by itself identify this measure with the inverse Fourier kernel. The next lemma provides the required identification and converts the measure positivity into pointwise positivity of the velocity fundamental solution.

\begin{lemma}[Positivity of the three-dimensional velocity kernel]\label{Lemma-Kernel-Identification}
	For every fixed $r>0$, the inverse Fourier representation of $G_1^{(3)}(t,\cdot)$ defines a function $K_r(t):=G_1^{(3)}(t,r)$ such that
	\begin{align*}
		K_r\in\ml C\bigl((0,\infty)\bigr)\cap L_{\mathrm{loc}}^1\bigl([0,\infty)\bigr).
	\end{align*}
	Moreover, for every $s>0$, the following identity is valid:
	\begin{align}\label{Kernel-Identification-Laplace}
		\int_0^\infty\mathrm{e}^{-st}K_r(t)\dd t=\frac{1}{4\pi r(1+s)}\mathrm{e}^{-\frac{rs}{\sqrt{1+s}}}.
	\end{align}
	Consequently, the nonnegative Radon measure $\mu_r$ obtained in Lemma~\ref{Lemma-Laplace-Kernel} is absolutely continuous with respect to the Lebesgue measure and satisfies $\mu_r(\mathrm{d} t)=K_r(t)\dd t$. In particular,
	\begin{align}\label{G3-Positive}
		G_1^{(3)}(t,r)\geqslant0
	\end{align}
	for every $t>0$ and $r>0$.
\end{lemma}

\begin{proof}
	Split the multiplier into its low-, middle- and high-frequency parts. Since the low- and middle-frequency multipliers are smooth and compactly supported in regard to the frequency variable, their inverse Fourier kernels are smooth for $t>0$. Moreover, the bounded-frequency estimate $|\widehat S_1(t,\rho)|\lesssim t$ for $0\leqslant t\leqslant1$ shows that
	\begin{align}\label{Low-Middle-Kernel-Small-Time}
		|K_r^{\low}(t)|+|K_r^{\midf}(t)|\lesssim_r t\ \ \text{for}\ \  0<t\leqslant1,
	\end{align}
	in which the inverse Fourier transform for radial functions is employed.
	
	For the high-frequency part, radial Fourier inversion expresses
	\begin{align}\label{High-Kernel-Radial}
		K_r^{\high}(t)=\frac{1}{2\pi^2r}\int_{R_0}^{\infty}\underbrace{\chi_{\high}(\rho)\widehat S_1(t,\rho)\rho}_{=:A(t,\rho)}\sin(r\rho)\dd\rho.
	\end{align}
     For $t\geqslant1$, Lemma~\ref{Lemma-High-Symbols} implies that $A(t,\cdot)$ is a symbol of order $-1$ and that $\partial_\rho A(t,\cdot)$ is integrable. For $0<t\leqslant1$, the estimates proved in Lemma~\ref{Lemma-High-Small-Time} give
	\begin{align*}
		|\partial_\rho A(t,\rho)|\lesssim\min\{t,\rho^{-2}\}+t\mathrm{e}^{-ct\rho^2}.
	\end{align*}
	For every compact interval $J\subset\subset(0,\infty)$, its right-hand side is bounded by an integrable function of $\rho$, uniformly for $t\in J$. Thanks to the boundary property that $A(t,\rho)$ vanishes at $\rho=R_0$ and tends to zero as $\rho\to\infty$, one integration by parts in \eqref{High-Kernel-Radial} allows
	\begin{align}\label{High-Kernel-IBP}
		K_r^{\high}(t)=\frac{1}{2\pi^2r^2}\int_{R_0}^{\infty}\partial_\rho A(t,\rho)\cos(r\rho)\dd\rho.
	\end{align}
	The preceding bounds and dominated convergence show that $K_r^{\high}\in\ml C\bigl((0,\infty)\bigr)$ for every
	fixed $r>0$. For $0<t\leqslant1$, Lemma~\ref{Lemma-High-Small-Time} further gives
	\begin{align*}
		\int_{R_0}^{\infty}|\partial_\rho A(t,\rho)|\dd\rho\lesssim\sqrt t.
	\end{align*}
	Therefore, \eqref{High-Kernel-IBP} implies
	\begin{align}\label{High-Kernel-Small-Time}
		|K_r^{\high}(t)|\lesssim_r\sqrt t \ \ \text{for}\ \  0<t\leqslant1.
	\end{align}
	Combining \eqref{Low-Middle-Kernel-Small-Time} and \eqref{High-Kernel-Small-Time}, we conclude our desired regularity for $K_r$.

For $R>R_0$, one may introduce the $R$-dependent kernel
\begin{align*}
	K_{r,R}(t):=\frac{1}{2\pi^2r}\int_0^R\widehat S_1(t,\rho)\rho\sin(r\rho)\dd\rho.
\end{align*}
Since the frequency integral is taken over a bounded interval, Fubini's theorem and the Laplace-transformed multiplier formula \eqref{Supp-01} with $\rho=|\zeta|$ give
\begin{align*}
	\int_0^\infty\mathrm{e}^{-st}K_{r,R}(t)\dd t=\frac{1}{2\pi^2r(1+s)}\int_0^R\frac{\rho\sin(r\rho)}{\rho^2+\frac{s^2}{1+s}}\dd\rho.
\end{align*}
Letting $R\to\infty$ and using the standard Fourier-sine transform formula
\begin{align*}
	\int_0^\infty\frac{\rho\sin(r\rho)}{\rho^2+\lambda^2}\dd\rho=\frac{\pi}{2}\mathrm{e}^{-\lambda r}\ \ \text{for}\ \  r>0,\ \lambda>0,
\end{align*}
we obtain
\begin{align}\label{Truncated-Laplace-Limit}
	\lim_{R\to\infty}\int_0^\infty\mathrm{e}^{-st}K_{r,R}(t)\dd t=\frac{1}{4\pi r(1+s)}\mathrm{e}^{-\frac{rs}{\sqrt{1+s}}}.
\end{align}
The estimates established above, including the boundary term arising from the truncation at $\rho=R$, yield uniformly in $R$ as follows:
\begin{align*}
	|K_{r,R}^{\high}(t)|&=\left|-\frac{A(t,R)\cos(rR)}{2\pi^2r^2}+\frac{1}{2\pi^2r^2}\int_{R_0}^{R}\partial_\rho A(t,\rho)\cos(r\rho)\dd\rho\right|\\
	&\lesssim_r |A(t,R)|+\int_{R_0}^{R}|\partial_\rho A(t,\rho)|\dd\rho\\
	&\lesssim_r
	\begin{cases}
		\sqrt{t}&\text{if}\ \ 0<t\leqslant1,\\
		\mathrm{e}^{-ct}&\text{if}\ \ t\geqslant1.
	\end{cases}
\end{align*}
The low- and middle-frequency kernels have at most polynomial growth in time. Therefore, for every fixed $s>0$, dominated convergence in \eqref{Truncated-Laplace-Limit} leads to \eqref{Kernel-Identification-Laplace}.

Consequently, the locally finite signed measure $K_r(t)\dd t$ and the nonnegative Radon measure $\mu_r$ have the same Laplace transform. The uniqueness of the Laplace transform (the uniqueness part of Bernstein's theorem \cite[Theorem~1.4]{Schilling-Song-Vondracek=2012}) yields $\mu_r(\mathrm{d}t)=K_r(t)\dd t$, which implies $K_r(t)\geqslant 0$ for almost every $t>0$. Since $K_r$ is continuous on $(0,\infty)$ and $\mu_r$ is nonnegative, we claim \eqref{G3-Positive}.
\end{proof}

\subsection{Dimension descent and the positive half-line kernel}

Let $G_1^{(1)}(t,z)$ denote the one-dimensional velocity fundamental solution. Its time-Laplace transform is given by
\begin{align}\label{Laplace-G1}
	\widetilde G_1^{(1)}(s,z)=\frac{1}{2s\sqrt{1+s}}\mathrm{e}^{-\frac{|z|s}{\sqrt{1+s}}},
\end{align}
in which we used \eqref{Supp-01} and the one-dimensional Yukawa kernel, whose deduction is completely parallel to that for the three-dimensional formula.

For later use, we note that, for every fixed $r>0$,
\begin{align*}
	G_1^{(1)}(\cdot,r),\int_r^\infty G_1^{(3)}(\cdot,\ell)\ell\dd\ell\in\mathcal{C}\bigl((0,\infty)\bigr)\cap L_{\mathrm{loc}}^1\bigl([0,\infty)\bigr).
\end{align*}
Indeed, the first assertion follows from the one-dimensional Fourier inversion formula, the bound
\begin{align*}
	|\widehat S_1(t,\rho)|\lesssim\min\{t,\rho^{-2}\},
\end{align*}
and dominated convergence, whereas the second follows from the frequency decomposition and the estimates used in the proof of Lemma~\ref{Lemma-Kernel-Identification}.

\begin{prop}[Dimension descent]\label{Prop-Dimension-Descent}
	For every $t>0$ and $r\geqslant0$, one has
	\begin{align}\label{Dimension-Descent}
		G_1^{(1)}(t,r)=2\pi\int_r^\infty G_1^{(3)}(t,\ell)\ell\dd\ell.
	\end{align}
	Consequently, for every fixed $t>0$, the function $r\mapsto G_1^{(1)}(t,r)$ is nonnegative and nonincreasing on $[0,\infty)$.
\end{prop}

\begin{proof}
	For $r>0$, the nonnegativity of $G_1^{(3)}$ and Tonelli's theorem gives
	\begin{align*}
		\mathcal{L}_{t\to s}\left[2\pi\int_r^\infty G_1^{(3)}(t,\ell)\ell\dd\ell\right](s,r)
		&=2\pi\int_r^\infty\widetilde G_1^{(3)}(s,\ell)\ell\dd\ell\\
		&=\frac{1}{2(1+s)}\int_r^\infty\mathrm{e}^{-\ell\psi(s)}\dd\ell\\
		&=\frac{1}{2s\sqrt{1+s}}\mathrm{e}^{-r\psi(s)},
	\end{align*}
	where we recall $\psi(s)=\frac{s}{\sqrt{1+s}}$. By \eqref{Laplace-G1}, the last expression is the Laplace transform of $G_1^{(1)}(t,r)$. The uniqueness of the Laplace transform for locally finite measures (see the uniqueness part of Bernstein's theorem \cite[Theorem~1.4]{Schilling-Song-Vondracek=2012}) proves \eqref{Dimension-Descent} for $r>0$, while the limit case $r=0$ follows by letting $r\to0^+$ and using monotone convergence. The nonnegativity and monotonicity with respect to $r$ follow automatically from \eqref{Dimension-Descent} and \eqref{G3-Positive}.
\end{proof}

For the half-line problem subject to the homogeneous Dirichlet boundary condition at the origin, the well-known method of odd extension allows to construct the Dirichlet kernel
\begin{align}\label{Dirichlet-Kernel}
	D(t,r,\rho):=G_1^{(1)}(t,r-\rho)-G_1^{(1)}(t,r+\rho)\ \ \text{for}\ \  t>0, \  r,\rho>0.
\end{align}
Indeed, if $f$ is defined on $(0,\infty)$ and $f_{\mathrm{odd}}$ denotes its odd extension to $\mathbb{R}$, then
\begin{align*}
	S_1(t)f_{\mathrm{odd}}(r)=\int_0^\infty D(t,r,\rho)f(\rho)\dd\rho\ \ \text{for}\ \  r>0.
\end{align*}
In particular, the evenness of $G_1^{(1)}(t,\cdot)$ and \eqref{Dimension-Descent} yield
\begin{align}\label{Dirichlet-Kernel-Integral}
	D(t,r,\rho)&=G_1^{(1)}(t,|r-\rho|)-G_1^{(1)}(t,r+\rho)\notag\\
	&=2\pi\int_{|r-\rho|}^{r+\rho}G_1^{(3)}(t,\ell)\ell\dd\ell\geqslant0.
\end{align}
This representation does not rely on any support property of the fundamental solution. This positivity replaces finite-speed support localization and is the key ingredient in the nonlinear lower-bound iteration.

\begin{remark}
	Although the dimension-descent identity has the same formal structure as those for the classical wave and heat kernels, its role in the present setting is quite different. For the wave equation, it is closely related to finite propagation and the singular behavior on the wave front, whereas for the heat equation positivity and radial monotonicity follow directly from the explicit Gaussian kernel. In the strongly damped wave case, the fundamental solution has infinite spatial support and its Fourier multiplier retains an oscillatory component. Therefore, dimension descent alone does not imply positivity. The radial monotonicity of $G_1^{(1)}$ with $r+\rho\geqslant |r-\rho|$ is the additional ingredient that yields
	\begin{align*}
		G_1^{(1)}(t,|r-\rho|)-G_1^{(1)}(t,r+\rho)\geqslant 0,
	\end{align*}
	and hence the positive half-line Dirichlet kernel required in the nonlinear lower-bound argument.
\end{remark}

\subsection{Uniform asymptotics in the moving wave-front region}

In this part, the moving wave-front region refers to the family of spatial shells
\begin{align*}
	r=t-\sigma\sqrt{t} \ \ \text{carrying}\ \ \sigma\in[\sigma_-,\sigma_+],
\end{align*}
where $0<\sigma_-<\sigma_+<\infty$ are fixed. That is to say, the observation point lies at a distance of order $\sqrt{t}$ inside the characteristic surface $r=t$. This is the region in which the low-frequency wave propagation and diffusive spreading balance, giving rise to the Gaussian profile $\exp\bigl(-\frac{\sigma^2}{2}\bigr)$. Because strong damping destroys finite propagation speed, the terminology ``wave-front'' is used here in an asymptotic sense and does not refer to the boundary of the support of the fundamental solution. 

We are going to identify the leading terms of the position and velocity kernels. We define the regular (in the sense of subtracting the Dirac singular part) position kernel via
\begin{align*}
	G_0^{(3),\reg}(t,x):=\ml F^{-1}_{\zeta\to x}\left[\widehat S_0(t,\zeta)-\mathrm{e}^{-t}\right].
\end{align*}
Namely, in the sense of tempered distributions where $\delta_0$ denotes the Dirac measure concentrated at the spatial origin, one notices that
\begin{align}\label{Position-Kernel-Decomposition}
	G_0^{(3)}(t,x)=\mathrm{e}^{-t}\delta_0+G_0^{(3),\reg}(t,x).
\end{align}
This definition is independent of the auxiliary frequency partition. For every fixed $t>0$, the symbol bounds in \cref{Appendix-Symbols} show that $G_0^{(3),\reg}(t,\cdot)$ is continuous on $\mb R^3\setminus\{0\}$.

\begin{prop}[Uniform wave-front asymptotic expansion]\label{Prop-Front-Asymptotics}
	Let $0<\sigma_-<\sigma_+<\infty$. Uniformly for $\sigma\in[\sigma_-,\sigma_+]$, one has, as $t\to\infty$,
	\begin{align}\label{Front-G1}
		G_1^{(3)}\bigl(t,t-\sigma\sqrt{t}\,\bigr)&=c_*t^{-\frac{3}{2}}\mathrm{e}^{-\frac{\sigma^2}{2}}+O(t^{-2}),\\
		\partial_rG_1^{(3)}\bigl(t,t-\sigma\sqrt{t}\,\bigr)&=c_*\sigma t^{-2}\mathrm{e}^{-\frac{\sigma^2}{2}}+O(t^{-\frac{5}{2}}),\label{Front-Dr-G1}\\
		G_0^{(3),\reg}\bigl(t,t-\sigma\sqrt{t}\,\bigr)&=-c_*\sigma t^{-2}\mathrm{e}^{-\frac{\sigma^2}{2}}+O(t^{-\frac{5}{2}}),\label{Front-G0}
	\end{align}
	where $c_*:=\frac{1}{4\pi\sqrt{2\pi}}$.
\end{prop}

\begin{proof}
	We compute the principal terms here and refer to \cref{Appendix-Front} for the uniform remainder estimates. The principal low-frequency multipliers are the diffusion-waves, namely,
	\begin{align*}
		m_{1,\mathrm P}(t,\rho):=\mathrm{e}^{-\frac{t\rho^2}{2}}\frac{\sin(t\rho)}{\rho}\ \ \text{and}\ \ m_{0,\mathrm P}(t,\rho):=\mathrm{e}^{-\frac{t\rho^2}{2}}\cos(t\rho).
	\end{align*}
	Concerning a radial multiplier $m=m(\rho)$ in dimension three, one may express its inverse Fourier transform via
	\begin{align*}
		\ml F^{-1}[m](r)=\frac{1}{2\pi^2r}\int_0^\infty m(\rho)\rho\sin(r\rho)\dd\rho.
	\end{align*}
	The product-to-sum identities and the Fourier transform of a Gaussian give
	\begin{align}\label{Principal-G1}
		G_{1,\mathrm P}^{(3)}(t,r)&=\frac{c_*}{r\sqrt{t}}\left(\mathrm{e}^{-\frac{(t-r)^2}{2t}}-\mathrm{e}^{-\frac{(t+r)^2}{2t}}\right),\\
		G_{0,\mathrm P}^{(3)}(t,r)&=\frac{c_*}{rt^{\frac{3}{2}}}\left((r+t)\mathrm{e}^{-\frac{(r+t)^2}{2t}}+(r-t)\mathrm{e}^{-\frac{(r-t)^2}{2t}}\right).\label{Principal-G0}
	\end{align}
	For instance, \eqref{Principal-G1} is deduced according to
	\begin{align*}
		G_{1,\mathrm P}^{(3)}(t,r)
		&=\frac{1}{2\pi^2r}\int_0^\infty\mathrm{e}^{-\frac{t\rho^2}{2}}\sin(t\rho)\sin(r\rho)\dd\rho\\
		&=\frac{1}{4\pi^2r}\int_0^\infty\mathrm{e}^{-\frac{t\rho^2}{2}}\left[\cos\bigl((t-r)\rho\bigr)-\cos\bigl((t+r)\rho\bigr)\right]\mathrm{d}\rho,
	\end{align*}
	associated with the Gaussian identity
\begin{align*}
	\int_0^\infty\mathrm{e}^{-\frac{t\rho^2}{2}}\cos\bigl((t\pm r)\rho\bigr)\dd\rho=\sqrt{\frac{\pi}{2t}}\,\mathrm{e}^{-\frac{(t\pm r)^2}{2t}}.
\end{align*}
	For $r=t-\sigma\sqrt{t}$, uniformly for $\sigma\in[\sigma_-,\sigma_+]$, one has
	\begin{align*}
		\mathrm{e}^{-\frac{(t-r)^2}{2t}}=\mathrm{e}^{-\frac{\sigma^2}{2}},\ \ \mathrm{e}^{-\frac{(t+r)^2}{2t}}=O(\mathrm{e}^{-ct}),\ \ \frac{1}{r}=\frac{1}{t}\bigl(1+O(t^{-\frac{1}{2}})\bigr).
	\end{align*}
	It follows from \eqref{Principal-G1} and \eqref{Principal-G0} that the principal kernels have the leading terms stated in \eqref{Front-G1} and \eqref{Front-G0}. Differentiating \eqref{Principal-G1} in regard to $r$, we also obtain
	\begin{align*}
		\partial_rG_{1,\mathrm P}^{(3)}(t,r)=-\frac{c_*}{r^2\sqrt{t}}\left(\mathrm{e}^{-\frac{(t-r)^2}{2t}}-\mathrm{e}^{-\frac{(t+r)^2}{2t}}\right)+\frac{c_*}{r\sqrt{t}}\left(\frac{t-r}{t}\mathrm{e}^{-\frac{(t-r)^2}{2t}}+\frac{t+r}{t}\mathrm{e}^{-\frac{(t+r)^2}{2t}}\right),
	\end{align*}
	which leads to \eqref{Front-Dr-G1} analogously.
	
	On the low-frequency support, we have the following asymptotic expansion:
	\begin{align}\label{Omega-Expansion}
		\omega(\rho)=\rho-\frac{\rho^3}{8}+O(\rho^5).
	\end{align}
	The uniform estimates in \cref{Appendix-Front} show that the difference between the exact low-frequency kernels and the principal kernels is $O(t^{-2})$ for $G_1^{(3)}$, and $O(t^{-\frac{5}{2}})$ for $\partial_rG_1^{(3)}$ as well as $G_0^{(3),\reg}$. The middle-frequency contribution is exponentially small by \eqref{Middle-Matrix}, while the regular high-frequency contribution is exponentially small with arbitrary spatial decay by \cref{Coro-High-Spatial}. These estimates are uniform for $\sigma\in[\sigma_-,\sigma_+]$ and complete the proof.
\end{proof}

The uniform expansion immediately yields a positive lower bound on every fixed compact subregion of the principal moving region.

\begin{coro}[Positivity in the moving wave-front region]\label{Coro-Positive-Front}
	Let $0<\sigma_-<\sigma_+<\infty$. There exist $T_{\mathrm{fr}},c_{\mathrm{fr}}>0$ such that
	\begin{align*}
		G_1^{(3)}(t,\ell)\geqslant c_{\mathrm{fr}}t^{-\frac{3}{2}}
	\end{align*}
	for $t\geqslant T_{\mathrm{fr}}$ and $\ell\in[t-\sigma_+\sqrt{t}, t-\sigma_-\sqrt{t}\,]$.
\end{coro}

\begin{proof}
	Write $\ell=t-\sigma\sqrt{t}$ with $\sigma\in[\sigma_-,\sigma_+]$. Due to the fact that
	\begin{align}\label{Eq-1}
		\inf_{\sigma\in[\sigma_-,\sigma_+]}\mathrm{e}^{-\frac{\sigma^2}{2}}=\mathrm{e}^{-\frac{\sigma_+^2}{2}}>0,
	\end{align}
	the uniform expansion \eqref{Front-G1}, together with the triangle inequality
	\begin{align*}
		G_1^{(3)}(t,\ell)\geqslant c_*t^{-\frac{3}{2}}\mathrm{e}^{-\frac{\sigma^2}{2}}-Ct^{-2}\geqslant \frac{c_*}{2}\mathrm{e}^{-\frac{\sigma_+^2}{2}}t^{-\frac{3}{2}}
	\end{align*}
	for suitably large time $t\geqslant T_{\mathrm{fr}}$ such that $t^{-2}=o(t^{-\frac{3}{2}})$, proves the assertion.
\end{proof}

The next proposition transfers the pointwise kernel expansion to compactly supported initial data without any radial symmetry assumption.

\begin{prop}[Positive linear lower bound]\label{Prop-Linear-Seed}
	Assume that $u_0,u_1\in\mathcal{C}_0^\infty$ and additionally that $M_1$ in \eqref{Positive-Mean} is positive. Define
	\begin{align*}
		W_{\lin}(t,r):=r\ml M\bigl[S_0(t)u_0+S_1(t)u_1\bigr](r).
	\end{align*}
	For every fixed $0<\sigma_-<\sigma_+<\infty$, one has
	\begin{align*}
		W_{\lin}\bigl(t,t-\sigma\sqrt{t}\,\bigr)=c_*M_1\mathrm{e}^{-\frac{\sigma^2}{2}}t^{-\frac{1}{2}}+O(t^{-1})\ \ \text{as}\ \ t\to\infty,
	\end{align*}
	uniformly for $\sigma\in[\sigma_-,\sigma_+]$. Consequently, there exist $T_{\lin},c_{\lin}>0$ such that
	\begin{align*}
		W_{\lin}(t,r)\geqslant c_{\lin}t^{-\frac{1}{2}}
	\end{align*}
	whenever $t\geqslant T_{\lin}$ and $r=t-\sigma\sqrt{t}$ for $\sigma\in[\sigma_-,\sigma_+]$.
\end{prop}

\begin{proof}
	Let the supports of the initial data be contained in $B_R$ with some $R>0$. Fix $x=r\omega$, $r=t-\sigma\sqrt{t}$ and $\sigma\in[\sigma_-,\sigma_+]$. For $|y|\leqslant R$, one has $\bigl||x-y|-r\bigr|\leqslant |y|\leqslant R$. Hence, after slightly enlarging the compact interval of admissible values of $\sigma$, all radii between $r$ and $|x-y|$ remain in the moving wave-front region. The mean value theorem and \eqref{Front-Dr-G1} therefore yield
	\begin{align*}
		G_1^{(3)}(t,|x-y|)=G_1^{(3)}(t,r)+O\bigl(|y|t^{-2}\bigr)
	\end{align*}
	uniformly in $\omega$, $\sigma$ and $|y|\leqslant R$. It follows that
	\begin{align*}
		S_1(t)u_1(x)&=G_1^{(3)}(t,r)\int_{B_R}u_1(y)\dd y+\int_{B_R}\left[G_1^{(3)}(t,|x-y|)-G_1^{(3)}(t,r)\right]u_1(y)\dd y\\
		&=M_1G_1^{(3)}(t,r)+O(t^{-2}).
	\end{align*}
	The uniform estimate \eqref{Front-G0} similarly implies that the regular part of $S_0(t)u_0(x)$ is $O(t^{-2})$. By \eqref{Position-Kernel-Decomposition}, its singular part is $\mathrm{e}^{-t}u_0(x)$, which vanishes for all sufficiently large $t$ because $r\to\infty$.
	
	Taking the spherical mean, multiplying by $r=t-\sigma\sqrt{t}$, and using \eqref{Front-G1}, we obtain
	\begin{align*}
		W_{\lin}\bigl(t,t-\sigma\sqrt{t}\,\bigr)&=r\left[M_1G_1^{(3)}(t,r)+O(t^{-2})\right]\\
		&=\bigl(t-\sigma\sqrt{t}\,\bigr)\left[c_*M_1\mathrm{e}^{-\frac{\sigma^2}{2}}t^{-\frac{3}{2}}+O(t^{-2})\right]\\
		&=c_*M_1\mathrm{e}^{-\frac{\sigma^2}{2}}t^{-\frac{1}{2}}+O(t^{-1})
	\end{align*}
	uniformly for $\sigma\in[\sigma_-,\sigma_+]$. Thanks to $M_1>0$ and \eqref{Eq-1}, the asserted lower bound follows for all sufficiently large $t$, i.e. $t\geqslant T_{\lin}$.
\end{proof}

\subsection{Kernel lower bounds and moment identities}

We introduce the following characteristic variables associated with a target point $(t,r)$ and a source point $(s,\rho)$:
\begin{align}\label{Characteristic-Variables}
	\alpha:=t+r,\ \ \beta:=t-r,\ \ \xi:=s+\rho,\ \ \eta:=s-\rho,\ \ \tau:=t-s,
\end{align}
which fulfill the relations
\begin{align}\label{Characteristic-Identities}
	\tau-(r-\rho)=\beta-\eta\ \ \text{and}\ \ r+\rho-\tau=\xi-\beta.
\end{align}

\begin{lemma}[A uniform positive lower bound for the Dirichlet kernel]\label{Lemma-Saturated-Kernel}
	There exist $c_D,T_D>0$ with the following property. Let $\tau\geqslant T_D$ and $r\geqslant\rho>0$. If $\beta-\eta\geqslant\frac{1}{4}\sqrt{\tau}$ and $\xi-\beta\geqslant-\frac{1}{8}\sqrt{\tau}$, then 
	\begin{align*}
	D(\tau,r,\rho)\geqslant c_D.
	\end{align*}
\end{lemma}

\begin{proof}
	By \eqref{Characteristic-Identities}, the first condition implies $r-\rho\leqslant\tau-\frac{1}{4}\sqrt{\tau}$,	whereas the second condition gives $r+\rho\geqslant\tau-\frac{1}{8}\sqrt{\tau}$, namely,
	\begin{align*}
		\left[\tau-\frac{1}{4}\sqrt{\tau},\tau-\frac{1}{8}\sqrt{\tau}\right]\subseteq[r-\rho,r+\rho].
	\end{align*}
	Corollary~\ref{Coro-Positive-Front} and \eqref{Dirichlet-Kernel-Integral} yield
	\begin{align*}
		D(\tau,r,\rho)\geqslant 2\pi c_{\mathrm{fr}}\int_{\tau-\frac{1}{4}\sqrt{\tau}}^{\tau-\frac{1}{8}\sqrt{\tau}}\tau^{-\frac{3}{2}}\ell\dd\ell=\pi c_{\mathrm{fr}}\left(\frac{1}{4}-\frac{3}{64\sqrt{\tau}}\right)\geqslant c_D
	\end{align*}
	for all sufficiently large $\tau$, i.e. $\tau\geqslant T_D$.
\end{proof}

The local ODE argument in the next sections requires only the total mass and the second moment of the one-dimensional kernel.

\begin{lemma}[Mass and second moment]\label{Lemma-Mass-Moment}
	For every $t>0$, one has
	\begin{align}\label{Mass-G1}
		\int_{\mb R}G_1^{(1)}(t,z)\dd z&=t,\\
		\int_{\mb R}z^2G_1^{(1)}(t,z)\dd z&=t^2+\frac{t^3}{3}.\label{Moment-G1}
	\end{align}
\end{lemma}

\begin{proof}
	According to Proposition~\ref{Prop-Dimension-Descent}, $G_1^{(1)}(t,\cdot)$ is a nonnegative, even and integrable kernel. Its Fourier transform solves \eqref{Linear-Fourier-ODE} with initial values $0$ and $1$. Evaluating at $\zeta=0$, we obtain $\widehat G_1^{(1)}(t,0)=t$, which proves \eqref{Mass-G1} obviously.
	
	To justify the second moment without assuming its finiteness in advance, we next use the symmetric second difference for $h\neq0$ as follows:
	\begin{align*}
		\frac{2}{h^2}\left(\widehat G_1^{(1)}(t,0)-\widehat G_1^{(1)}(t,h)\right)&=\int_{\mb R}\frac{2\bigl(1-\cos(hz)\bigr)}{h^2}
		G_1^{(1)}(t,z)\dd z\\
		&=\int_{\mb R}z^2\left(\frac{\sin\left(\frac{hz}{2}\right)}{\frac{hz}{2}}\right)^2G_1^{(1)}(t,z)\dd z.
	\end{align*}
	Near $\zeta=0$, the multiplier satisfies
	\begin{align*}
		\widehat G_1^{(1)}(t,\zeta)=t-\left(\frac{t^2}{2}+\frac{t^3}{6}\right)\zeta^2+O(\zeta^4).
	\end{align*}
	This expansion follows by inserting a power series in $\zeta^2$ into \eqref{Linear-Fourier-ODE}. Fatou's lemma first shows that the second moment is finite. Since the squared sinc factor is bounded by $1$, dominated convergence subsequently yields
	\begin{align*}
		\int_{\mb R}z^2G_1^{(1)}(t,z)\dd z=-\partial_\zeta^2\widehat G_1^{(1)}(t,0)=t^2+\frac{t^3}{3}.
	\end{align*}
	This proves \eqref{Moment-G1}.
\end{proof}

We shall also need rapid decay of the homogeneous solution in every strict interior region. Since its proof requires a separate nonstationary-phase analysis of the low-frequency kernels, together with the treatment of the middle- and high-frequency components, whose computation is elementary but tedious, we defer it to \cref{Appendix-Interior}.

\begin{lemma}[Rapid decay in strict interior regions]\label{Lemma-Interior-Decay}
	Let $0<a<b<1$, and let $u_0,u_1\in\mathcal{C}_0^\infty$. For every $N>0$, there exist $C_N,T_N>0$ such that
	\begin{align}\label{Supp-07}
		\sup_{at\leqslant r\leqslant bt}\left|W_{\lin}(t,r)\right|\leqslant C_Nt^{-N}\ \ \text{for}\ \ t\geqslant T_N.
	\end{align}
\end{lemma}

\section{Spherical reduction and the nonlinear lower-bound inequality}\label{Section-Spherical}

It is well-known that spherical means and nonlinear lower-bound iterations are classical for semilinear wave equations, see \cite{John=1979,Agemi-Kurokawa-Takamura=2000,Zhou=2001,Takamura-Wakasa=2011} and the references therein. Although finite propagation speed is lost under the strong damping $-\Delta u_t$, Jensen's inequality preserves the full nonlinearity after spherical averaging, while positivity of the half-line Duhamel kernel yields the required lower bound.

For a function $f=f(x)$ on $\mb R^3$, motivated by the classic three-dimensional radial waves transformation, we introduce the next quantity:
\begin{align*}
	(\ml Rf)(r):=r\ml M[f](r)\ \ \text{for}\ \ 	r>0.
\end{align*}

\begin{lemma}[The radialization map]\label{Lemma-Radialization}
	The map $\ml R$ extends boundedly from $L^2(\mb R^3)$ to $L^2(0,\infty)$ and from $H^2(\mb R^3)$ to $H^2(0,\infty)\cap H_0^1(0,\infty)$ satisfying
	\begin{align}\label{Radialization-L2}
		\|\ml Rf\|_{L^2(0,\infty)}&\leqslant\frac{1}{\sqrt{4\pi}}\|f\|_{L^2(\mb R^3)},\\
        \|\ml Rf\|_{H^2(0,\infty)}&\leqslant C\|f\|_{H^2(\mb R^3)}.\notag
	\end{align}
	Moreover, one has
	\begin{align}\label{Radialization-Laplacian}
		(\ml Rf)''=\ml R(\Delta f)
	\end{align}
	in the sense of distributions on $(0,\infty)$. The odd extension of $\ml Rf$ belongs to $H^2(\mb R)$.
\end{lemma}

\begin{proof}
	Let us consider $f\in\mathcal{C}_0^\infty(\mb R^3)$. By Jensen's inequality and polar coordinates, one derives
	\begin{align*}
		\|\ml Rf\|_{L^2(0,\infty)}^2
		&=\int_0^\infty r^2|\ml M[f](r)|^2\dd r\\
		&\leqslant\frac{1}{4\pi}\int_0^\infty\int_{\mb S^2}|f(r\theta)|^2r^2\dd\theta\dd r=\frac{1}{4\pi}\|f\|_{L^2(\mb R^3)}^2,
	\end{align*}
	which finishes \eqref{Radialization-L2}. The spherical-mean identity $\ml M[\Delta f](r)=\partial_r^2\ml M[f](r)+\frac{2}{r}\partial_r\ml M[f](r)$ gives
	\begin{align*}
		\partial_r^2\bigl(r\ml M[f](r)\bigr)=r\ml M[\Delta f](r)\ \ \text{for}\ \ r>0,
	\end{align*}
	and hence \eqref{Radialization-Laplacian}.
	
	Let $(\ml Rf)_{\mathrm{odd}}$ be the odd extension of $\ml Rf$ to $\mb R$. It is smooth and compactly supported. Applying \eqref{Radialization-L2} to $f$ and $\Delta f$, we estimate
	\begin{align*}
		\|(\ml Rf)_{\mathrm{odd}}\|_{L^2(\mb R)}+\|(\ml Rf)_{\mathrm{odd}}''\|_{L^2(\mb R)}\leqslant C\|f\|_{H^2(\mb R^3)}.
	\end{align*}
	Moreover, integration by parts gives
	\begin{align*}
		\|(\ml Rf)_{\mathrm{odd}}'\|_{L^2(\mb R)}^2\leqslant\|(\ml Rf)_{\mathrm{odd}}\|_{L^2(\mb R)}\|(\ml Rf)_{\mathrm{odd}}''\|_{L^2(\mb R)}\leqslant C^2\|f\|_{H^2(\mb R^3)}^2.
	\end{align*}
	Consequently, the last two estimates yield
	\begin{align*}
		\|(\ml Rf)_{\mathrm{odd}}\|_{H^2(\mb R)}\leqslant C\|f\|_{H^2(\mb R^3)}.
	\end{align*}
	
	The density of $\mathcal{C}_0^\infty(\mb R^3)$ in $H^2(\mb R^3)$ therefore yields the bounded extension of $\ml R$. The corresponding odd extensions converge in $H^2(\mb R)$ and remain odd. Their restrictions have zero trace at the origin and hence belong to $H_0^1(0,\infty)$. Finally, \eqref{Radialization-Laplacian} passes to the limit in the sense of distributions.
\end{proof}

For $r\geqslant0$, one now writes
\begin{align*}
	w(t,r):=\bigl(\ml R[u(t,\cdot)]\bigr)(r)=r\ml M[u(t,\cdot)](r).
\end{align*}
By the assumed solution regularity and the Sobolev–Morrey embedding $H^2(\mb R^3)\hookrightarrow\mathcal{C}(\mb R^3)$, the function $w$ is continuous in $(t,r)$ on compact subsets of the maximal existence interval. Moreover, Lemma~\ref{Lemma-Radialization} shows that the odd extension of $w(t,\cdot)$ belongs to $H^2(\mb R)$ and satisfies $w(t,0)=0$.

By recalling \eqref{Radialization-Laplacian} and applying $\ml{R}$ to the equation for $u$, one may deduce
\begin{align}\label{Supp-02}
	w_{tt}-w_{rr}-w_{rrt}=F(t,r)\ \ \text{for}\ \ r>0,
\end{align}
carrying $F(t,r):=r\ml M[|u(t,\cdot)|^p](r)$. The odd extension of $w$ satisfies the corresponding equation on $\mb R$ without any additional distribution supported at the origin.
\begin{remark}
Let $w_{\mathrm{odd}}$ and $F_{\mathrm{odd}}$ denote the odd extensions of $w$ and $F$, respectively. Since $w_{\mathrm{odd}}(t,\cdot)\in H^2(\mb R)$, its first spatial derivative has no jump at the origin, so no Dirac mass supported at $r=0$ arises under the second spatial derivative or its distributional time derivative. Consequently, the equation
\begin{align*}
	\partial_t^2w_{\mathrm{odd}}-\partial_r^2w_{\mathrm{odd}}-\partial_t\partial_r^2w_{\mathrm{odd}}=F_{\mathrm{odd}}
\end{align*}
holds in $\mathcal{D}'\bigl((0,T)\times\mb R\bigr)$.
\end{remark}
\noindent Finally, Jensen's inequality can deal with the nonlinearity in the next way:
\begin{align*}
	F(t,r)\geqslant r\bigl|\ml M[u](t,r)\bigr|^p=r^{1-p}|w(t,r)|^p\ \ \text{for}\ \ r>0.
\end{align*}

\begin{prop}[Half-line representation for Sobolev solutions]\label{Prop-Half-Line-Representation}
	Let $u$ be a mild Sobolev solution on $[0,T]$ arising from compactly supported smooth initial data, where $0<T<T_{\max}$. Then the solution and the source to \eqref{Supp-02} satisfy
	\begin{align*}
		w&\in\mathcal{C}\bigl([0,T],H^2(0,\infty)\cap H_0^1(0,\infty)\bigr)\cap\mathcal{C}^1\bigl([0,T],L^2(0,\infty)\bigr),\\
		F&\in\mathcal{C}\bigl([0,T],L^2(0,\infty)\bigr).
	\end{align*}
	Moreover, for every $0\leqslant t\leqslant T$ and $r>0$, one has
	\begin{align}\label{Fundamental-Lower}
		w(t,r)\geqslant W_{\lin}(t,r)+\int_0^t\int_0^\infty D(t-s,r,\rho)\rho^{1-p}|w(s,\rho)|^p\dd\rho\dd s.
	\end{align}
\end{prop}

\begin{proof}
	Lemma~\ref{Lemma-Radialization} and the regularity of the mild solution claim the asserted regularity of $w$. Moreover, the Sobolev embedding $H^2(\mb R^3)\hookrightarrow L^\infty(\mb R^3)$ and the continuity of the power nonlinearity imply that
	\begin{align*}
		|u|^p\in\mathcal{C}\bigl([0,T],L^2(\mb R^3)\bigr)\ \ \Rightarrow\ \ F\in\mathcal{C}\bigl([0,T],L^2(0,\infty)\bigr).
	\end{align*}

	Let $S_j^{(1),D}(t)$ denote the restriction to $(0,\infty)$ of the one-dimensional propagator acting on odd extensions. For $f\in\mathcal{C}_0^\infty(\mb R^3)$, \eqref{Radialization-Laplacian} shows that the odd extension of $\ml R\bigl(S_j(t)f\bigr)$ satisfies
	\begin{align*}
		\ml z_{tt}-\ml z_{rr}-\ml z_{rrt}=0\ \ \text{on}\ \ (0,\infty)\times\mb R
	\end{align*}
	with initial data 
	\begin{align*}
	&(\ml z,\ml z_t)(0,r)=\bigl( (\ml Rf)_{\mathrm{odd}}(r),0\bigr)\ \ \text{for}\ \ j=0,\\
	&(\ml z,\ml z_t)(0,r)=\bigl(0,(\ml Rf)_{\mathrm{odd}}(r)\bigr)\ \ \text{for}\ \ j=1.
	\end{align*}
	 Uniqueness for the one-dimensional linear Cauchy problem therefore yields
	\begin{align}\label{Propagator-Intertwining}
		\ml R\bigl(S_j(t)f\bigr)=S_j^{(1),D}(t)\ml Rf\ \ \text{for}\ \ j\in\{0,1\}.
	\end{align}
	By density, Lemma~\ref{Lemma-Radialization}, and the bounded-time estimates for the linear propagators, \eqref{Propagator-Intertwining} extends to $f\in H^2(\mb R^3)$ for $j=0$ and to $f\in L^2(\mb R^3)$ for $j=1$.

Applying \eqref{Propagator-Intertwining} to \eqref{Mild-Formula}, we obtain the half-line Duhamel formula in $H^2(0,\infty)\cap H_0^1(0,\infty)$. The method of images and \eqref{Dirichlet-Kernel} then give
\begin{align*}
w(t,r)=W_{\lin}(t,r)+\int_0^t\int_0^\infty D(t-s,r,\rho)F(s,\rho)\dd\rho\dd s.
\end{align*}
Since $H^2(0,\infty)\hookrightarrow\mathcal{C}^1\bigl([0,\infty)\bigr)$, this identity holds pointwise. By Jensen's inequality,
\begin{align}\label{Supp-03}
	F(s,\rho)=\rho\ml M|u(s,\cdot)|^p\geqslant\rho^{1-p}|w(s,\rho)|^p.
\end{align}
Moreover, the asserted regularity and $w(s,0)=0$ imply $w(s,\rho)=O(\rho)$ as $\rho\to0^+$, uniformly for $s\in[0,T]$. Hence, recalling \eqref{Supp-03}, $\rho^{1-p}|w(s,\rho)|^p=O(\rho)$ near the origin. The positivity of $D$ now yields \eqref{Fundamental-Lower} to complete our proof.
\end{proof}

By \eqref{Dirichlet-Kernel-Integral}, the kernel $D$ is nonnegative. However, no global sign condition is imposed on $W_{\lin}$. We therefore retain the linear term in \eqref{Fundamental-Lower} until its positivity is established in the relevant region.

In the characteristic variables introduced in \eqref{Characteristic-Variables}, inspired by the operator decomposition $\partial_t^2-\partial_r^2=4\partial_{\alpha}\partial_{\beta}$, we define
\begin{align*}
	\ml W(\alpha,\beta):=w\left(\frac{\alpha+\beta}{2},\frac{\alpha-\beta}{2}\right).
\end{align*}
Under the corresponding change of variables \eqref{Characteristic-Variables}, one has
\begin{align}\label{Characteristic-Jacobian}
	\dd s\dd\rho=\frac{1}{2}\dd\xi\dd\eta.
\end{align}
These coordinates will be used in the subcritical and critical blow-up arguments.

\section{Blow-up in the subcritical range $2<p<p_{\mathrm{crit}}$}\label{Section-Subcritical}

\subsection{Parabolic iteration and interior propagation}

For $B>0$, we define the following parabolic region:
\begin{align*}
	\ml P_B:=\bigl\{(\alpha,\beta):\ \beta\geqslant B\ \ \text{and}\ \ \beta^2\leqslant\alpha\leqslant2\beta^2\bigr\}.
\end{align*}

\begin{prop}[First nonlinear lower bound]\label{Prop-First-Lower}
	Let $\nu_0:=\frac{3p-5}{2}$. There exist $A_0,B_0>0$ such that
	\begin{align}\label{First-Lower}
		\ml W(\alpha,\beta)\geqslant A_0\beta^{-\nu_0}\ \ \text{for}\ \ (\alpha,\beta)\in\ml P_{B_0}.
	\end{align}
\end{prop}

\begin{proof}
	Fix $(\alpha,\beta)\in\ml P_B$ and take $B$ sufficiently large. Note that \eqref{Fundamental-Lower} with the aid of \eqref{Characteristic-Jacobian} can be rewritten as
	\begin{align*}
		\ml W(\alpha,\beta)&\geqslant W_{\lin}\left(\frac{\alpha+\beta}{2},\frac{\alpha-\beta}{2}\right)\\
		&\quad+\frac{1}{2}\iint_{\ml D_{\alpha,\beta}}D\left(\frac{\alpha+\beta-\xi-\eta}{2},\frac{\alpha-\beta}{2},\frac{\xi-\eta}{2}\right)\left(\frac{\xi-\eta}{2}\right)^{1-p}|\ml W(\xi,\eta)|^p\dd\eta\dd\xi,
	\end{align*}
	where
	\begin{align*}
		\ml D_{\alpha,\beta}:=\big\{(\xi,\eta):\ 0<\xi+\eta<\alpha+\beta\ \ \text{and}\ \ 0<\xi-\eta\big\}.
	\end{align*}
	 The source region is chosen so that the source points lie in the positive moving-front region of Proposition~\ref{Prop-Linear-Seed}, while the characteristic margins remain comparable to $\sqrt{\tau}$, as required by Lemma~\ref{Lemma-Saturated-Kernel}. More precisely, we retain
	 \begin{align*}
	 	2\beta\leqslant\xi\leqslant4\beta\ \ \text{and}\ \ \sqrt{\xi}\leqslant\eta\leqslant\frac{6}{5}\sqrt{\xi}.
	 \end{align*}
	This region lies in the characteristic integration domain since $\eta<\beta<\xi<\alpha$ for large $\beta$. Moreover,
	\begin{align*}
		s=\frac{\xi+\eta}{2}\approx\xi,\ \ \rho=\frac{\xi-\eta}{2}\approx\xi,\ \ \frac{\eta}{\sqrt{s}}\approx1.
	\end{align*}
	Proposition~\ref{Prop-Linear-Seed} and the positivity of the nonlinear term imply
	\begin{align}\label{Source-Seed}
		w(s,\rho)\geqslant W_{\lin}(s,\rho)\geqslant c_{\lin} s^{-\frac{1}{2}}\approx c\xi^{-\frac{1}{2}}.
	\end{align}
	
	We next verify that the kernel is uniformly positive throughout the retained source region. Due to $(\alpha,\beta)\in\ml P_B$, whereas $\xi\approx\beta$ and $\eta\approx\sqrt{\beta}$, we derive
	\begin{align*}
		\tau=t-s=\frac{\alpha+\beta-\xi-\eta}{2}\approx\alpha\approx\beta^2.
	\end{align*}
	In particular, $\sqrt{\tau}\approx\beta$. Additionally,
	\begin{align*}
		r-\rho=\frac{\alpha-\beta-\xi+\eta}{2}>0
	\end{align*}
	for all sufficiently large $\beta$, and hence $r\geqslant\rho$. Notice that
	\begin{align*}
		\beta-\eta&\geqslant\beta-C\sqrt{\beta}\gtrsim\beta\approx\sqrt{\tau},\\
		\xi-\beta&\geqslant\beta\approx\sqrt{\tau}.
	\end{align*}
	So, all the assumptions of Lemma~\ref{Lemma-Saturated-Kernel} are satisfied, and consequently $D(\tau,r,\rho)\geqslant c_D$. Furthermore, $\frac{\beta}{\sqrt{t}}$ remains in a fixed compact subset of $(0,\infty)$ on $\ml P_B$, and thus Proposition~\ref{Prop-Linear-Seed} implies $W_{\lin}(t,r)\geqslant0$ for sufficiently large $B$.
	
	Using \eqref{Source-Seed} and $\rho\approx\xi$, we obtain
	\begin{align*}
		\ml W(\alpha,\beta)&\gtrsim\int_{2\beta}^{4\beta}\int_{\sqrt{\xi}}^{\frac{6}{5}\sqrt{\xi}}\xi^{1-p}\xi^{-\frac{p}{2}}\dd\eta\dd\xi\\
		&=\frac{2}{5(3p-5)}\left((2\beta)^{-\frac{3p-5}{2}}-(4\beta)^{-\frac{3p-5}{2}}\right)\gtrsim\beta^{\frac{5}{2}-\frac{3p}{2}}.
	\end{align*}
	This proves \eqref{First-Lower}.
\end{proof}

\begin{prop}[Parabolic-region iteration]\label{Prop-Parabolic-Iteration}
	Assume that, for some $A_j,B_j>0$ and $\nu_j\in\mb R$,
	\begin{align}\label{Iteration-Hypothesis}
		\ml W(\xi,\eta)\geqslant A_j\eta^{-\nu_j}
	\end{align}
	whenever
	\begin{align}\label{Source-Parabolic-Region}
		\eta\geqslant B_j\ \ \text{and}\ \ \eta^2\leqslant\xi\leqslant2\eta^2.
	\end{align}
	Put $q_j:=p\nu_j+2p-5$. Then there exist $A_{j+1},B_{j+1}>0$ such that
	\begin{align*}
		\ml W(\alpha,\beta)\geqslant A_{j+1}\times
		\begin{cases}
			\beta^{-\frac{q_j}{2}}&\text{if}\ \  q_j>0,\\
			\log\beta&\text{if}\ \  q_j=0,\\
			\beta^{-q_j}&\text{if}\ \  q_j<0,
		\end{cases}
	\end{align*}
	for $(\alpha,\beta)\in\ml P_{B_{j+1}}$. In particular, if $q_j>0$, then \eqref{Iteration-Hypothesis} holds at the next step with $\nu_{j+1}:=\frac{q_j}{2}$.
\end{prop}

\begin{proof}
	Let us fix $(\alpha,\beta)\in\ml P_B$. The source region is chosen so that it lies inside the parabolic region where the induction hypothesis is valid, while retaining uniform positive margins for the Dirichlet kernel. More precisely, we now retain
	\begin{align}\label{Iteration-Source}
		4\sqrt{\beta}\leqslant\eta\leqslant\frac{\beta}{8}\ \ \text{and}\ \ \eta^2\leqslant\xi\leqslant\frac{5}{4}\eta^2.
	\end{align}
	Taking $B$ sufficiently large, this region is contained in \eqref{Source-Parabolic-Region}. Moreover, since $\alpha\approx\beta^2$, $\eta\leqslant\frac{\beta}{8}$ and $\xi\leqslant\frac{5\beta^2}{256}$, we have $\tau\approx\beta^2$.
	The characteristic margins satisfy
	\begin{align*}
		\beta-\eta&\gtrsim\beta\approx\sqrt{\tau},\\
		\xi-\beta&\geqslant\eta^2-\beta\gtrsim\beta,
	\end{align*}
	while $r-\rho\gtrsim\beta^2$. Hence, Lemma~\ref{Lemma-Saturated-Kernel} concludes $D(\tau,r,\rho)\geqslant c_D$. As in the proof of Proposition~\ref{Prop-First-Lower}, the homogeneous term is nonnegative for all sufficiently large $B$.
	
	On \eqref{Iteration-Source}, we have $\rho\approx\eta^2$. Therefore,
	\begin{align*}
		\ml W(\alpha,\beta)&\gtrsim A_j^p\int_{4\sqrt{\beta}}^{\frac{\beta}{8}}\int_{\eta^2}^{\frac{5}{4}\eta^2}\eta^{2(1-p)}\eta^{-p\nu_j}\dd\xi\dd\eta\\
		&\gtrsim A_j^p\int_{4\sqrt{\beta}}^{\frac{\beta}{8}}\eta^{-q_j-1}\dd\eta\gtrsim A_j^p\times \begin{cases}
			\beta^{-\frac{q_j}{2}}&\text{if}\ \  q_j>0,\\
			\log\beta&\text{if}\ \  q_j=0,\\
			\beta^{-q_j}&\text{if}\ \  q_j<0,
		\end{cases}
	\end{align*}
	for all sufficiently large $\beta$. The final conclusion follows with $A_{j+1}:=c_jA_j^p$ after increasing $B_{j+1}$ if necessary.
\end{proof}

Starting from $\nu_0:=\frac{3p-5}{2}$, we apply Proposition~\ref{Prop-Parabolic-Iteration} successively as long as $q_j>0$. Then
\begin{align*}
	\nu_{j+1}=\frac{p}{2}\nu_j+\frac{2p-5}{2}\ \ \Rightarrow\ \ \nu_j=\frac{p(3p-7)}{2(p-2)}\left(\frac{p}{2}\right)^j-\frac{2p-5}{p-2}.
\end{align*}
Since $2<p<\frac{7}{3}$, we have $\nu_j\to-\infty$, and therefore $q_j\leqslant0$ after finitely many steps. At the first such step, Proposition~\ref{Prop-Parabolic-Iteration} yields a lower bound for $\ml W(\alpha,\beta)$ that grows logarithmically in $\beta$ when $q_j=0$ and algebraically when $q_j<0$. Then by choosing $B_*>0$ sufficiently large, there exist $A_*>0$ such that
\begin{align}\label{Parabolic-Positive}
	\ml W(\alpha,\beta)\geqslant A_*\ \ \text{for}\ \ (\alpha,\beta)\in\ml P_{B_*}.
\end{align}
\begin{remark}
At the critical power $p=\frac{7}{3}$, this finite-step improvement disappears. Indeed, $\nu_j\equiv1$ and $q_j\equiv2$, so the preceding power iteration never yields a nondecaying lower bound. This is why a refined logarithmic iteration is required in the critical case.
\end{remark}

Having obtained a nondecaying lower bound for $\ml W$ in the parabolic region, we now return to the original $(t,r)$-variables and propagate this bound to a strict interior region, where it yields a time-growing lower bound for $w$.

\begin{prop}[Interior growth]\label{Prop-Interior-Growth}
	Let $\mu_p:=\frac{5}{2}-p>0$. There exist $T_{\mathrm{int}},C_{\mathrm{int}}>0$ such that
	\begin{align}\label{Interior-Growth}
		w(t,r)\geqslant C_{\mathrm{int}}t^{\mu_p}
	\end{align}
	for $t\geqslant T_{\mathrm{int}}$ and $\frac{2}{5}t\leqslant r\leqslant\frac{3}{4}t$
\end{prop}

\begin{proof}
	Fix $t$ sufficiently large and $\frac{2}{5}t\leqslant r\leqslant\frac{3}{4}t$. In \eqref{Fundamental-Lower}, let us restrict the integration to
	\begin{align}\label{Supp-04}
		\Omega_{\mathrm{sub}}(t):=\left\{(\xi,\eta): \ \frac{3}{4}t\leqslant\xi\leqslant\frac{4}{5}t\ \ \text{and}\ \ \sqrt{\frac{3\xi}{5}}\leqslant\eta\leqslant\sqrt{\frac{4\xi}{5}}\right\},
	\end{align}
	whose reason and motivation are similar to those in the last two propositions.
	On this region, one knows
	\begin{align*}
		\frac{5}{4}\eta^2\leqslant\xi\leqslant\frac{5}{3}\eta^2\ \ \text{and}\ \ \eta\approx\sqrt{t}.
	\end{align*}
	Hence, for sufficiently large $t$, the source point belongs to $\ml P_{B_*}$ and \eqref{Parabolic-Positive} gives $\ml W(\xi,\eta)\geqslant A_*$.
	
	Let us take $\beta:=t-r$. Then $\frac{1}{4}t\leqslant\beta\leqslant\frac{3}{5}t$. Moreover, throughout the source region,
	\begin{align*}
		s=\frac{\xi+\eta}{2}\approx t,\ \ \rho=\frac{\xi-\eta}{2}\approx t\ \ \text{and}\ \ \tau=t-s\approx t.
	\end{align*}
	Since $\rho\leqslant\frac{2}{5}t\leqslant r$, we have $r\geqslant\rho$. Furthermore, $\beta-\eta\gtrsim t$ and $\xi-\beta\geqslant\frac{3}{20}t$. Lemma~\ref{Lemma-Saturated-Kernel} therefore yields $D(t-s,r,\rho)\geqslant c_D$.
	
	The $\xi$-interval has length comparable to $t$, while the $\eta$-interval has length comparable to $\sqrt{t}$. Since $\rho\approx t$, \eqref{Fundamental-Lower} and \eqref{Characteristic-Jacobian} imply
	\begin{align*}
		w(t,r)&\geqslant W_{\lin}(t,r)+cA_*^p\int_{\frac{3}{4}t}^{\frac{4}{5}t}\int_{\sqrt{\frac{3\xi}{5}}}^{\sqrt{\frac{4\xi}{5}}}\rho^{1-p}\dd\eta\dd\xi\\
		&\geqslant W_{\lin}(t,r)+ct^{\frac{5}{2}-p}.
	\end{align*}
	By Lemma~\ref{Lemma-Interior-Decay}, we conclude
	\begin{align*}
		\sup_{\frac{2}{5}t\leqslant r\leqslant\frac{3}{4}t}|W_{\lin}(t,r)|\leqslant C_Nt^{-N}
	\end{align*}
	for every $N>0$. Recalling $\frac{5}{2}-p>0$, the linear term can be absorbed for sufficiently large $t$, which proves \eqref{Interior-Growth}.
\end{proof}

\subsection{A short-time kernel mass bound}

The final iteration requires the following lower bound over short time intervals. For $T>0$, set
\begin{align*}
	J_T:=\left[\frac{T}{2},\frac{2T}{3}\right].
\end{align*}

\begin{lemma}[Short-time kernel mass]\label{Lemma-Local-Kernel-Mass}
	There exist $\kappa_0\in\left(0,\frac{1}{100}\right)$ and $T_0,c_0>0$ such that
	\begin{align}\label{Local-Kernel-Mass}
		\inf_{r\in J_T}\int_{J_T}D(\tau,r,\rho)\dd\rho\geqslant c_0\tau
	\end{align}
	whenever $T\geqslant T_0$ and $0<\tau\leqslant\kappa_0T$.
\end{lemma}

\begin{proof}
	We may set the length $L_T:=|J_T|=\frac{T}{6}$. Since $G_1^{(1)}(\tau,\cdot)$ is nonnegative, even and nonincreasing on $(0,\infty)$, for every $r\in J_T$ we have
	\begin{align*}
		\int_{J_T}G_1^{(1)}(\tau,r-\rho)\dd\rho&=\int_{r-\frac{2T}{3}}^{r-\frac{T}{2}}G_1^{(1)}(\tau,z)\dd z\\
		&\geqslant\int_0^{L_T}G_1^{(1)}(\tau,z)\dd z=\frac{1}{2}\left(\tau-\int_{|z|>L_T}G_1^{(1)}(\tau,z)\dd z\right),
	\end{align*}
	in which we used $\int_{\mb R}G_1^{(1)}(\tau,z)\dd z=\tau$.
Recalling again $G_1^{(1)}(\tau,\cdot)$ is nonnegative, Chebyshev's inequality and \eqref{Moment-G1} give
\begin{align*}
	\int_{|z|>L_T}G_1^{(1)}(\tau,z)\dd z\leqslant\frac{1}{L_T^2}\int_{\mb R}z^2G_1^{(1)}(\tau,z)\dd z=\frac{\tau^2+\frac{\tau^3}{3}}{L_T^2},
\end{align*}
which indicates
\begin{align*}
	\int_{J_T}G_1^{(1)}(\tau,r-\rho)\dd\rho\geqslant\frac{\tau}{2}-\frac{\tau^2+\frac{\tau^3}{3}}{2L_T^2}.
\end{align*}
	On the other hand, $r+\rho\geqslant T$ for $r,\rho\in J_T$, and hence
	\begin{align*}
		\int_{J_T}G_1^{(1)}(\tau,r+\rho)\dd\rho\leqslant\int_T^\infty G_1^{(1)}(\tau,z)\dd z\leqslant\frac{\tau^2+\frac{\tau^3}{3}}{2T^2}.
	\end{align*}
	It follows from \eqref{Dirichlet-Kernel} and $L_T=\frac{T}{6}$ that
	\begin{align*}
		\int_{J_T}D(\tau,r,\rho)\dd\rho\geqslant\frac{\tau}{2}-C\frac{\tau^2+\tau^3}{T^2}.
	\end{align*}
	If $0<\tau\leqslant\kappa_0T$, then $\frac{\tau^2+\tau^3}{T^2}\leqslant\left(\frac{\kappa_0}{T}+\kappa_0^2\right)\tau$. Choosing $\kappa_0>0$ sufficiently small and then $T_0>0$ sufficiently large, we obtain
	\begin{align*}
		\int_{J_T}D(\tau,r,\rho)\dd\rho\geqslant\frac{\tau}{4}
	\end{align*}
	uniformly for $r\in J_T$, $T\geqslant T_0$ and $0<\tau\leqslant\kappa_0T$. This proves \eqref{Local-Kernel-Mass} with $c_0=\frac{1}{4}$.
\end{proof}

\subsection{Completion of the subcritical proof}

\begin{proof}[Proof of Theorem~\ref{Thm-Blowup} for $2<p<\frac{7}{3}$]
	Assume, to the contrary, that the maximal mild Sobolev solution is global in time. Let $\kappa_0$ be given by Lemma~\ref{Lemma-Local-Kernel-Mass}. For $T$ sufficiently large, set
	\begin{align}\label{Supp-05}
		\ml Q_T:=[T,(1+\kappa_0)T]\times J_T.
	\end{align}
	Since $\kappa_0<\frac{1}{100}$, every $(t,r)\in\ml Q_T$ satisfies $\frac{2}{5}t\leqslant r\leqslant\frac{3}{4}t$. Hence, Proposition~\ref{Prop-Interior-Growth} addresses
	\begin{align}\label{Subcritical-Baseline}
		w(t,r)\geqslant A_T:=C_{\mathrm{int}}T^{\mu_p}\ \ \text{for}\ \ (t,r)\in\ml Q_T.
	\end{align}
	
	For $(t,r)\in\ml Q_T$, one can consider the source region defined in characteristic variables by \eqref{Supp-04}. Throughout this region,
	\begin{align*}
		s=\frac{\xi+\eta}{2}\leqslant\frac{2}{5}t+C\sqrt{t}<\frac{T}{2}
	\end{align*}
	for all sufficiently large $T$. It is therefore disjoint from $[T,t]\times J_T$. Repeating the estimate in the proof of Proposition~\ref{Prop-Interior-Growth}, we immediately arrive at
	\begin{align}\label{Subcritical-Early-Contribution}
		W_{\lin}(t,r)+\iint_{\Omega_{\mathrm{sub}}(t)}D(t-s,r,\rho)\rho^{1-p}|w(s,\rho)|^p\dd\rho\dd s\geqslant A_T,
	\end{align}
	where the time-dependent domain $\Omega_{\mathrm{sub}}(t)$ defined in \eqref{Supp-04} is equivalent to 
	\begin{align*}
		\Omega_{\mathrm{sub}}(t)=\left\{(s,\rho):\ \frac{3}{4}t\leqslant s+\rho\leqslant\frac{4}{5}t\ \ \text{and}\ \  \sqrt{\frac{3(s+\rho)}{5}}\leqslant s-\rho\leqslant\sqrt{\frac{4(s+\rho)}{5}}\right\}.
	\end{align*}

Let us now define
\begin{align*}
	h_T(t):=\inf_{r\in J_T}w(t,r)\ \ \text{for}\ \ T\leqslant t\leqslant(1+\kappa_0)T.
\end{align*}
The joint continuity of $w$ and the compactness of $J_T$ imply that $h_T\in\mathcal{C}\bigl([T,(1+\kappa_0)T]\bigr)$. Moreover, \eqref{Subcritical-Baseline} gives 
\begin{align*}
h_T(t)\geqslant A_T
\end{align*}
on this interval.

For $(t,r)\in\ml Q_T$, we retain, in addition to the early contribution in \eqref{Subcritical-Early-Contribution}, the region $T\leqslant s\leqslant t$ and $\rho\in J_T$. Thanks to the half-line Dirichlet kernel $D(t-s,r,\rho)$ being nonnegative and $\rho^{1-p}\gtrsim T^{1-p}$ on $J_T$, \eqref{Fundamental-Lower} yields
\begin{align*}
	w(t,r)\geqslant A_T+cT^{1-p}\int_T^t [h_T(s)]^p\int_{J_T}D(t-s,r,\rho)\dd\rho\dd s.
\end{align*}
Lemma~\ref{Lemma-Local-Kernel-Mass} is applicable because $0\leqslant t-s\leqslant\kappa_0T$. Taking the infimum over $r\in J_T$, the nonlinear integral inequality arises
\begin{align}\label{Subcritical-Volterra}
	h_T(t)\geqslant A_T+cT^{1-p}\int_T^t(t-s)[h_T(s)]^p\dd s.
\end{align}

Let $y$ be the maximal solution of
\begin{align*}
	y''=cT^{1-p}y^p \ \ \text{with}\ \ y(T)=A_T,\ y'(T)=0.
\end{align*}
A standard Volterra comparison applied to \eqref{Subcritical-Volterra} gives $	h_T(t)\geqslant y(t)$ throughout their common interval of existence. Multiplying the equation for $y$ by $y'$ shows that its blow-up delay $\tau_{\mathrm b}$ satisfies
\begin{align*}
	\tau_{\mathrm b}\lesssim T^{\frac{p-1}{2}}A_T^{-\frac{p-1}{2}}.
\end{align*}
Recalling $A_T=cT^{\mu_p}$, we obtain
\begin{align*}
	\frac{\tau_{\mathrm b}}{T}\lesssim T^{\frac{(p-1)(1-\mu_p)-2}{2}}.
\end{align*}
Since $\mu_p=\frac{5}{2}-p$ and $2<p<\frac{7}{3}$, one has
\begin{align*}
	(p-1)(1-\mu_p)-2=(p-1)\left(p-\frac{3}{2}\right)-2<0.
\end{align*}
Therefore, for sufficiently large $T$, one has
\begin{align*}
	\tau_{\mathrm b}<\kappa_0T.
\end{align*}
Thus, $y$ blows up before $(1+\kappa_0)T$, whereas
\begin{align*}
	h_T\in\mathcal{C}\bigl([T,(1+\kappa_0)T]\bigr)
	\ \ \mbox{and}\ \ 
	h_T(t)\geqslant y(t).
\end{align*}
This contradiction completes the proof.
\end{proof}

\section{Blow-up at the critical exponent $p=p_{\mathrm{crit}}$}\label{Section-Critical}

\subsection{A kernel lower bound for separated times}	

For $t>0$, let us define the interval
\begin{align}\label{Gamma-t}
	\Gamma_t:=\left[t-\frac{9}{8}\sqrt{t},t-\sqrt{t}\right].
\end{align}
Proposition~\ref{Prop-Linear-Seed} yields $T_0,a_0>0$ such that
\begin{align}\label{Critical-Seed}
	w(t,r)\geqslant a_0t^{-\frac{1}{2}}
\end{align}
for $t\geqslant T_0$ and $r\in\Gamma_t$.

\begin{lemma}[Separated-time kernel lower bound]\label{Lemma-Separated-Time}
	There exist $\kappa\in\left(0,\frac{1}{100}\right)$, $T_1\geqslant T_0$ and $c_1>0$ such that
	\begin{align}\label{Separated-Time-Lower}
		D(t-s,r,\rho)\geqslant c_1\frac{s}{\sqrt{t}}
	\end{align}
	whenever $t\geqslant T_1$, $T_0\leqslant s\leqslant\kappa\sqrt{t}$, $r\in\Gamma_t$ and $\rho\in\Gamma_s$.
\end{lemma}

\begin{proof}
	Let $\tau:=t-s$. By \eqref{Gamma-t}, one obtains
	\begin{align}\label{Critical-Beta-Eta}
		\sqrt{t}\leqslant\beta\leqslant\frac{9}{8}\sqrt{t}\ \ \text{and}\ \ \sqrt{s}\leqslant\eta\leqslant\frac{9}{8}\sqrt{s},
	\end{align}
	while
	\begin{align}\label{Critical-Xi}
		2s-\frac{9}{8}\sqrt{s}\leqslant\xi\leqslant2s-\sqrt{s}.
	\end{align}
	Moreover, for $s\leqslant\kappa\sqrt{t}$, we get
	\begin{align*}
		r-\rho\geqslant t-\frac{9}{8}\sqrt{t}-s\geqslant t-\left(\frac{9}{8}+\kappa\right)\sqrt{t}>0
	\end{align*}
	for all sufficiently large $t$. Hence, $r\geqslant\rho$ after increasing $T_1$ if necessary.

For $\ell\in[r-\rho,r+\rho]$, the characteristic identities give
\begin{align*}
	\beta-\xi=\tau-(r+\rho)\leqslant\tau-\ell\leqslant\tau-(r-\rho)=\beta-\eta.
\end{align*}
By \eqref{Critical-Beta-Eta} and \eqref{Critical-Xi},
\begin{align*}
	\beta-\xi&\geqslant\sqrt{t}-2s\geqslant(1-2\kappa)\sqrt{t},\\
	\beta-\eta&\leqslant\frac{9}{8}\sqrt{t}.
\end{align*}
Moreover, $t-\kappa\sqrt{t}\leqslant\tau\leqslant t$, and hence $\sqrt{\tau}\approx\sqrt{t}$ uniformly for $s\leqslant\kappa\sqrt{t}$. Choosing $\kappa>0$ sufficiently small and then increasing $T_1$, we obtain
\begin{align}\label{Entire-Front-Region}
	\frac{3}{4}\sqrt{\tau}\leqslant\tau-\ell\leqslant\frac{3}{2}\sqrt{\tau}\ \ \text{for}\ \ \ell\in[r-\rho,r+\rho].
\end{align}
 Via \eqref{Entire-Front-Region}, the interval $[r-\rho,r+\rho]$ lies in a fixed positive region of Corollary~\ref{Coro-Positive-Front} implying
 \begin{align*}
 	G_1^{(3)}(\tau,\ell)\geqslant c\tau^{-\frac{3}{2}}\ \ \text{for}\ \ \ell\in[r-\rho,r+\rho].
 \end{align*}
After increasing $T_1$ if necessary, one finds
 \begin{align*}
 	\ell\geqslant\tau-\frac{3}{2}\sqrt{\tau}\geqslant\frac{\tau}{2}\ \ \text{and}\ \ \tau\approx t.
 \end{align*}
 Since $s\geqslant T_0$ and $\rho\geqslant s-\frac{9}{8}\sqrt{s}$, we may also assume that $\rho\geqslant\frac{s}{2}$. Therefore, \eqref{Dirichlet-Kernel-Integral} yields
 \begin{align*}
 	D(\tau,r,\rho)&=2\pi\int_{r-\rho}^{r+\rho}G_1^{(3)}(\tau,\ell)\ell\dd\ell\geqslant c\tau^{-\frac{1}{2}}\rho.
 \end{align*}
 This proves \eqref{Separated-Time-Lower}.
 \end{proof}

\subsection{Critical accumulation and iteration on moving shells}

Let us define the continuous function on $[T_0,\infty)$ as follows:
\begin{align*}
	A_{\mathrm{crit}}(t):=\sqrt{t}\inf_{r\in\Gamma_t}w(t,r)\ \ \text{for}\ \ t\geqslant T_0.
\end{align*}
Indeed, under the parametrization $r=t-\sigma\sqrt{t}$ with $\sigma\in[1,\frac{9}{8}]$, the function \begin{align*}
(t,\sigma)\mapsto\sqrt{t}\,w(t,t-\sigma\sqrt{t}\,)
\end{align*} 
is continuous on a fixed compact parameter interval. Moreover, \eqref{Critical-Seed} gives
\begin{align}\label{A-Seed}
	A_{\mathrm{crit}}(t)\geqslant a_0\ \ \text{for}\ \ t\geqslant T_0.
\end{align}

\begin{prop}[Critical accumulation inequality]\label{Prop-Critical-Accumulation}
	There exist $c_+>0$ and $T_2\geqslant T_1$ such that
	\begin{align}\label{Critical-Accumulation}
		A_{\mathrm{crit}}(t)\geqslant a_0+c_+\int_{T_0}^{\kappa\sqrt{t}}\frac{[A_{\mathrm{crit}}(s)]^{\frac{7}{3}}}{s}\dd s\ \ \text{for}\ \ t\geqslant T_2.
	\end{align}
\end{prop}

\begin{proof}
	Take $T_2\geqslant T_1$ sufficiently large such that $\kappa\sqrt{t}\geqslant T_0$ for $t\geqslant T_2$. Fix $t\geqslant T_2$ and $r\in\Gamma_t$. Restricting the integral in \eqref{Fundamental-Lower} to
	\begin{align*}
		T_0\leqslant s\leqslant\kappa\sqrt{t}\ \ \text{and}\ \ \rho\in\Gamma_s,
	\end{align*}
	Proposition~\ref{Prop-Linear-Seed} and Lemma~\ref{Lemma-Separated-Time} conclude
	\begin{align*}
		w(t,r)
		\geqslant a_0t^{-\frac{1}{2}}+ct^{-\frac{1}{2}}\int_{T_0}^{\kappa\sqrt{t}}s\int_{\Gamma_s}\rho^{-\frac{4}{3}}|w(s,\rho)|^{\frac{7}{3}}\dd\rho\dd s.
	\end{align*}
	For $\rho\in\Gamma_s$, the definition of $A_{\mathrm{crit}}(t)$ and \eqref{A-Seed} imply
	\begin{align*}
		w(s,\rho)\geqslant s^{-\frac{1}{2}}A_{\mathrm{crit}}(s)>0.
	\end{align*}
	From $\rho\leqslant s$ and $|\Gamma_s|=\frac{\sqrt{s}}{8}$, we obtain
	\begin{align*}
		s\int_{\Gamma_s}\rho^{-\frac{4}{3}}|w(s,\rho)|^{\frac{7}{3}}\dd\rho\geqslant c\frac{[A_{\mathrm{crit}}(s)]^{\frac{7}{3}}}{s}.
	\end{align*}
	Consequently,
	\begin{align*}
		w(t,r)\geqslant t^{-\frac{1}{2}}\left(a_0+c\int_{T_0}^{\kappa\sqrt{t}}\frac{[A_{\mathrm{crit}}(s)]^{\frac{7}{3}}}{s}\dd s\right).
	\end{align*}
	Taking the infimum over $r\in\Gamma_t$ and multiplying by $\sqrt{t}$ proves \eqref{Critical-Accumulation}.
\end{proof}

The following iteration is a variant of the classical slicing method \cite{Agemi-Kurokawa-Takamura=2000,Wakasa-Yordanov=2019,Palmieri-Takamura=2019}, adapted here to the square-root delay in \eqref{Critical-Accumulation}. By \eqref{A-Seed} and \eqref{Critical-Accumulation}, we can estimate
\begin{align*}
A_{\mathrm{crit}}(t)\geqslant a_0+c\log\left(\frac{\kappa\sqrt{t}}{T_0}\right)
\end{align*} 
for all sufficiently large $t$, and hence $A_{\mathrm{crit}}(t)\to\infty$.

Choose 
\begin{align*}
	M_0>\max\left\{1,c_+^{-\frac{1}{p-1}}\right\}.
\end{align*} 
Since $A_{\mathrm{crit}}(t)\to\infty$, we may take $T_0^\sharp\geqslant T_2$ sufficiently large such that
\begin{align*}
	A_{\mathrm{crit}}(t)\geqslant M_0\ \ \text{for}\ \ t\geqslant T_0^\sharp.
\end{align*}
We next choose recursively
\begin{align*}
	T_{j+1}^\sharp:=\left(\frac{\mathrm{e}T_j^\sharp}{\kappa}\right)^2\ \ \text{and}\ \ M_{j+1}:=c_+M_j^p.
\end{align*}
We claim that
\begin{align}\label{Critical-Discrete-Iteration}
	A_{\mathrm{crit}}(t)\geqslant M_j\ \ \text{for}\ \ t\geqslant T_j^\sharp.
\end{align}
Indeed, if \eqref{Critical-Discrete-Iteration} holds at the $j$-th step and $t\geqslant T_{j+1}^\sharp$, then $\kappa\sqrt{t}\geqslant\mathrm{e}T_j^\sharp$, as well as
\begin{align*}
	A_{\mathrm{crit}}(t)\geqslant c_+\int_{T_j^\sharp}^{\mathrm{e}T_j^\sharp}\frac{[A_{\mathrm{crit}}(s)]^p}{s}\dd s\geqslant c_+M_j^p=M_{j+1},
\end{align*}
which proves the claim by induction.

Let $d_0:=\log M_0+\frac{\log c_+}{p-1}>0$. The two recurrences deduce that
\begin{align*}
	\log M_j&=d_0p^j-\frac{\log c_+}{p-1},\\
	\log T_j^\sharp&=2^j\left(\log T_0^\sharp+2\log\frac{\mathrm{e}}{\kappa}\right)-2\log\frac{\mathrm{e}}{\kappa}.
\end{align*}
We may set $\theta:=\log_2p=\log_2\frac{7}{3}>1$. For sufficiently large $t$, choose $j$ such that $T_j^\sharp\leqslant t<T_{j+1}^\sharp$. Then \eqref{Critical-Discrete-Iteration} yields
\begin{align*}
	\log A_{\mathrm{crit}}(t)\geqslant cp^j=c(2^j)^\theta\geqslant c(\log t)^\theta,
\end{align*}
and
\begin{align}\label{Critical-Shell-Growth}
	w(t,r)\geqslant t^{-\frac{1}{2}}\exp\left(c(\log t)^\theta\right)
\end{align}
for $t\geqslant T_3$ and $r\in\Gamma_t$.

\subsection{Interior propagation and ODE comparison}

\begin{prop}[Critical interior growth]\label{Prop-Critical-Interior}
	There exist $T_{\mathrm{crit}},C_{\mathrm{crit}},c_{\mathrm{crit}}>0$ such that
	\begin{align}\label{Critical-Interior-Growth}
		w(t,r)\geqslant C_{\mathrm{crit}}t^{-1}\exp\left(c_{\mathrm{crit}}(\log t)^\theta\right)
	\end{align}
	for $t\geqslant T_{\mathrm{crit}}$ and $\frac{2}{5}t\leqslant r\leqslant\frac{3}{4}t$.
\end{prop}

\begin{proof}
	Fix $t$ sufficiently large and $\frac{2}{5}t\leqslant r\leqslant\frac{3}{4}t$. In \eqref{Fundamental-Lower}, let us restrict the integration to
	\begin{align}\label{Supp-06}
		\frac{3}{8}t\leqslant s\leqslant\frac{2}{5}t\ \ \text{and}\ \ \rho\in\Gamma_s.
	\end{align}
	On this region, $s\approx\rho\approx t$ and $\tau=t-s\approx t$. Moreover, recalling that $\beta=t-r$, $\eta=s-\rho$ and $\xi=s+\rho$, we have
	\begin{align*}
		\beta-\eta&\geqslant\frac{t}{4}-C\sqrt{t}\gtrsim t,\\
		\xi-\beta&\geqslant\frac{3}{4}t-C\sqrt{t}-\frac{3}{5}t\gtrsim t,
	\end{align*}
	whereas $\rho\leqslant s-\sqrt{s}<\frac{2}{5}t\leqslant r$. Hence, Lemma~\ref{Lemma-Saturated-Kernel} gives $D(t-s,r,\rho)\geqslant c$ throughout the retained region.
	
	According to \eqref{Critical-Shell-Growth}, we are able to deduce
	\begin{align*}
		|w(s,\rho)|^{\frac{7}{3}}\geqslant s^{-\frac{7}{6}}\exp\left(c(\log s)^\theta\right)\ \ \text{for}\ \ \rho\in\Gamma_s.
	\end{align*}
	Since $|\Gamma_s|=\frac{\sqrt{s}}{8}$ and $(\log s)^\theta\geqslant c(\log t)^\theta$ on $\frac{3}{8}t\leqslant s\leqslant\frac{2}{5}t$, \eqref{Fundamental-Lower} eventually yields
	\begin{align*}
		w(t,r)&\geqslant W_{\lin}(t,r)+c\int_{\frac{3}{8}t}^{\frac{2}{5}t}s^{-\frac{4}{3}}s^{-\frac{7}{6}}\exp\left(c(\log s)^\theta\right)|\Gamma_s|\dd s\\
		&\geqslant W_{\lin}(t,r)+ct^{-1}\exp\left(c_1(\log t)^\theta\right).
	\end{align*}
	Lemma~\ref{Lemma-Interior-Decay} shows that the homogeneous term is negligible compared with the positive term for sufficiently large $t$. This proves \eqref{Critical-Interior-Growth}.
\end{proof}

\begin{proof}[Proof of Theorem~\ref{Thm-Blowup} for $p=\frac{7}{3}$]
	Assume, to the contrary, that the maximal mild Sobolev solution is global in time. Let $\kappa_0$ be given by Lemma~\ref{Lemma-Local-Kernel-Mass}. For $T$ sufficiently large, $\ml Q_T$ is defined by \eqref{Supp-05}.
	Since $\kappa_0<\frac{1}{100}$, the cylinder $\ml Q_T$ lies in the interior region of Proposition~\ref{Prop-Critical-Interior} so that
	\begin{align*}
		w(t,r)\geqslant \widetilde{A}_T:=CT^{-1}\exp\left(c(\log T)^\theta\right)\ \ \text{for}\ \ (t,r)\in\ml Q_T.
	\end{align*}
	
	For $(t,r)\in\ml Q_T$, let us consider the source region \eqref{Supp-06} used in the proof of Proposition~\ref{Prop-Critical-Interior}. Thanks to
	\begin{align*}
		s\leqslant\frac{2}{5}(1+\kappa_0)T<\frac{T}{2},
	\end{align*}
	this region is disjoint from $[T,t]\times J_T$. The same proof therefore yields
	\begin{align*}
		W_{\lin}(t,r)+\int_{\frac{3}{8}t}^{\frac{2}{5}t}\int_{\Gamma_s}D(t-s,r,\rho)\rho^{-\frac{4}{3}}|w(s,\rho)|^{\frac{7}{3}}\dd\rho\dd s\geqslant \widetilde{A}_T.
	\end{align*}
	
	Let us define
	\begin{align*}
		h_T(t):=\inf_{r\in J_T}w(t,r)\ \ \text{for}\ \ T\leqslant t\leqslant(1+\kappa_0)T.
	\end{align*}
	Then $h_T\in\mathcal{C}\bigl([T,(1+\kappa_0)T]\bigr)$ and $h_T(t)\geqslant \widetilde{A}_T$. Retaining in addition the region $T\leqslant s\leqslant t$ and $\rho\in J_T$ in \eqref{Fundamental-Lower}, we obtain
	\begin{align*}
		w(t,r)\geqslant \widetilde{A}_T+cT^{-\frac{4}{3}}\int_T^t [h_T(s)]^{\frac{7}{3}}\int_{J_T}D(t-s,r,\rho)\dd\rho\dd s.
	\end{align*}
	Indeed, $\rho^{-\frac{4}{3}}\gtrsim T^{-\frac{4}{3}}$ on $J_T$, while $w(s,\rho)\geqslant h_T(s)>0$. Lemma~\ref{Lemma-Local-Kernel-Mass} applies since $0\leqslant t-s\leqslant\kappa_0T$. Taking the infimum over $r\in J_T$ yields
	\begin{align}\label{Critical-Volterra}
		h_T(t)\geqslant \widetilde{A}_T+cT^{-\frac{4}{3}}\int_T^t(t-s)[h_T(s)]^{\frac{7}{3}}\dd s.
	\end{align}
	
	Let $y$ be the maximal solution of
	\begin{align*}
		y''=cT^{-\frac{4}{3}}y^{\frac{7}{3}}\ \ \text{with}\ \ y(T)=\widetilde{A}_T,\ y'(T)=0.
	\end{align*}
	A standard Volterra comparison applied to \eqref{Critical-Volterra} gives $h_T(t)\geqslant y(t)$ throughout their common interval of existence. Multiplying the equation for $y$ by $y'$ shows that its blow-up delay $\tau_{\mathrm b}$ satisfies
	\begin{align*}
		\tau_{\mathrm b}\lesssim T^{\frac{2}{3}}\widetilde{A}_T^{-\frac{2}{3}}.
	\end{align*}
	Consequently,
	\begin{align*}
		\frac{\tau_{\mathrm b}}{T}\lesssim T^{-\frac{1}{3}}\widetilde{A}_T^{-\frac{2}{3}}\lesssim T^{\frac{1}{3}}\exp\left(-c(\log T)^\theta\right)\to0\ \ \text{as}\ \ T\to\infty
	\end{align*}
	because of $\theta>1$. Thus $\tau_{\mathrm b}<\kappa_0T$ for sufficiently large $T$. The comparison solution then blows up before $(1+\kappa_0)T$, contradicting the continuity of $h_T$ on $[T,(1+\kappa_0)T]$. This completes the proof.
\end{proof}

\appendix

\section{Middle- and high-frequency multiplier bounds}
\label{Appendix-Symbols}

We collect the middle- and high-frequency multiplier estimates used throughout the linear and pointwise kernel analysis.

At high frequencies, we denote
\begin{align}\label{ab-Definition}
	a(\rho):=-\lambda_+(\rho)\ \ \text{and}\ \ b(\rho):=-\lambda_-(\rho).
\end{align}
As $\rho\to\infty$, one has the following asymptotic behavior:
\begin{align}\label{ab-Asymptotics}
	a(\rho)=1+O(\rho^{-2}),\ \ b(\rho)=\rho^2-1+O(\rho^{-2})\ \ \text{and}\ \ b(\rho)-a(\rho)\approx\rho^2.
\end{align}
For the middle frequencies, the matrix
\begin{align*}
	\mathbf A(\rho):=\begin{pmatrix}
		0&1\\
		-\rho^2&-\rho^2
	\end{pmatrix}
\end{align*}
has spectrum in a fixed left half-plane. At the double root $\rho=2$, the corresponding polynomial factor is bounded by $1+t$ and can be absorbed into a weaker exponential decay. Consequently, there exist constants $c>0$ and $C>0$ such that
\begin{align}\label{Middle-Matrix}
	\sup_{\rho\in\supp\chi_{\midf}}\|\mathrm{e}^{t\mathbf A(\rho)}\|\leqslant C\mathrm{e}^{-ct}\ \ \text{for every}\  \ t\geqslant0.
\end{align}

\begin{lemma}[Uniform middle-frequency multiplier bounds]\label{Lemma-Middle-Multiplier}
	Let $I\subset\subset(0,\infty)$ be a compact interval containing $\supp\chi_{\midf}$. There exists $c_I>0$ such that, for every $m\in\mb N_0$, there exists $C_m>0$ satisfying
	\begin{align*}
		\sup_{\rho\in I}\left\|\partial_\rho^m\mathrm{e}^{t\mathbf A(\rho)}\right\|\leqslant C_m\,\mathrm{e}^{-c_I t}
	\end{align*}
	for every $t\geqslant0$.
\end{lemma}

\begin{proof}
	Because of $I\subset\subset(0,\infty)$, the characteristic roots of $\mathbf A(\rho)$ remain uniformly in the open left half-plane. The double root at $\rho=2$ produces only a polynomial factor in $t$. Standard differentiation estimates for parameter-dependent matrix semigroups therefore lead to
	\begin{align*}
		\sup_{\rho\in I}\left\|\partial_\rho^m\mathrm{e}^{t\mathbf A(\rho)}\right\|\leqslant C_m(1+t)^m\mathrm{e}^{-2c_I t} \leqslant C_m\,\mathrm{e}^{-c_I t}.
	\end{align*}
	Our proof is complete.
\end{proof}

The high-frequency root splitting and the corresponding symbol structure are standard for structurally damped evolution equations (see, for instance, \cite[Section~8]{DAbbicco-Ebert=2022}).

\begin{lemma}[High-frequency root bounds]\label{Lemma-High-Roots}
	Let $a$ and $b$ be defined in \eqref{ab-Definition}. For every $k\in\mb N_0$, there exists $C_k>0$ such that
	\begin{align*}
		\left|\partial_\rho^k\bigl(a(\rho)-1\bigr)\right|&\leqslant C_k\rho^{-2-k},\\
		\left|\partial_\rho^k b(\rho)\right|&\leqslant C_k\rho^{2-k},\\
		\left|\partial_\rho^k\frac{1}{b(\rho)-a(\rho)}\right|&\leqslant C_k\rho^{-2-k},
	\end{align*}
	for every $\rho\geqslant R_0$. Moreover, $1\leqslant a(\rho)\leqslant 2$ and $b(\rho)\geqslant\frac{\rho^2}{2}$.
\end{lemma}

\begin{lemma}[Small-time high-frequency bounds]\label{Lemma-High-Small-Time}
	For $0\leqslant t\leqslant1$ and $\rho\geqslant R_0$, one has
	\begin{align}\label{High-Small-Time-Bounds}
		\rho^2|\widehat S_1(t,\rho)|+|\partial_t\widehat S_1(t,\rho)|+|\widehat S_0(t,\rho)|+|\partial_t\widehat S_0(t,\rho)|\leqslant C
	\end{align}
	and
	\begin{align*}
		\int_{R_0}^\infty\left|\partial_\rho\bigl(\chi_{\high}(\rho)\rho\widehat S_1(t,\rho)\bigr)\right|\mathrm{d}\rho\leqslant C\sqrt{t}.
	\end{align*}
\end{lemma}

\begin{proof}
	The difference formula gives
	\begin{align*}
		\widehat S_1(t,\rho)=t\int_0^1\exp\left(-t\bigl[(1-\vartheta)a(\rho)+\vartheta b(\rho)\bigr]\right)\mathrm{d}\vartheta.
	\end{align*}
	Since $a(\rho)\geqslant1$ and $b(\rho)-a(\rho)\approx\rho^2$, it follows that $|\widehat S_1(t,\rho)|\lesssim\min\{t,\rho^{-2}\}$. The explicit formulas for $\widehat S_0$ and $\widehat S_1$, together with $\partial_t\widehat S_0(t,\rho)=-\rho^2\widehat S_1(t,\rho)$ yield \eqref{High-Small-Time-Bounds}.
	
    According to $q(\rho):=\frac{\rho}{b(\rho)-a(\rho)}$, by Lemma~\ref{Lemma-High-Roots}, one claims
	\begin{align*}
		|q(\rho)|\lesssim\rho^{-1},\ \ |q'(\rho)|\lesssim\rho^{-2},\ \ |a'(\rho)|\lesssim\rho^{-3},\ \ |b'(\rho)|\lesssim\rho.
	\end{align*}
	Using the next identity:
	\begin{align*}
		\rho\widehat S_1(t,\rho)=q(\rho)\left(\mathrm{e}^{-a(\rho)t}-\mathrm{e}^{-b(\rho)t}\right),
	\end{align*}
	we estimate
	\begin{align*}
		\left|\partial_\rho\bigl(\rho\widehat S_1(t,\rho)\bigr)\right|\lesssim\min\left\{t,\rho^{-2}\right\}+t\mathrm{e}^{-ct\rho^2}.
	\end{align*}
	The term involving $\chi_{\high}'$ is supported in a fixed compact interval and contributes $O(t)$. For $0<t\leqslant1$, we estimate
	\begin{align*}
		\int_{R_0}^\infty\min\left\{t,\rho^{-2}\right\}\mathrm{d}\rho+t\int_{R_0}^\infty\mathrm{e}^{-ct\rho^2}\mathrm{d}\rho\lesssim\sqrt{t}.
	\end{align*}
	Since the assertion is immediate for $t=0$, the proof is complete.
\end{proof}

\begin{lemma}[Large-time high-frequency symbol bounds]\label{Lemma-High-Symbols}
	There exists $c>0$ such that, for every multi-index $\gamma$ and every $j\in\{0,1\}$, there exists $C_{\gamma,j}>0$ satisfying
	\begin{align*}
		\left|\partial_\zeta^\gamma\partial_t^j\bigl(\chi_{\high}(|\zeta|)\widehat S_1(t,\zeta)\bigr)\right|&\leqslant C_{\gamma,j}\,\mathrm{e}^{-ct}\langle\zeta\rangle^{-2-|\gamma|},\\\left|\partial_\zeta^\gamma\partial_t^j\bigl(\chi_{\high}(|\zeta|)\bigl(\widehat S_0(t,\zeta)-\mathrm{e}^{-t}\bigr)\bigr)\right|
		&\leqslant C_{\gamma,j}\,\mathrm{e}^{-ct}\langle\zeta\rangle^{-2-|\gamma|},
	\end{align*}
	for every $t\geqslant1$. The bounded-time estimates needed below are provided separately by Lemma~\ref{Lemma-High-Small-Time}.
\end{lemma}

\begin{proof}
	Put $\rho:=|\zeta|$. By Lemma~\ref{Lemma-High-Roots} and the Fa\`a di Bruno formula, for every $k,N\in\mb N_0$,
	\begin{align*}
		\left|\partial_\rho^k\mathrm{e}^{-a(\rho)t}\right|&\leqslant C_k\,\mathrm{e}^{-ct}\rho^{-k},\\
		\left|\partial_\rho^k\mathrm{e}^{-b(\rho)t}\right|&\leqslant C_{k,N}\,\mathrm{e}^{-ct}\rho^{-N},
	\end{align*}
	for $t\geqslant1$ and $\rho\geqslant R_0$. Indeed, the polynomial factors in $t$ produced by differentiation are absorbed by the exponential decay, while $b(\rho)\geqslant\frac{\rho^2}{2}$ gives arbitrary negative order for the fast-root term. Employing additionally
	\begin{align*}
		\widehat S_1(t,\zeta)=\frac{\mathrm{e}^{-a(\rho)t}-\mathrm{e}^{-b(\rho)t}}{b(\rho)-a(\rho)},
	\end{align*}
	the first asserted estimate follows directly from Lemma~\ref{Lemma-High-Roots} and the Leibniz rule.
	
	For the first multiplier, we are able to write
	\begin{align*}
		\widehat S_0(t,\zeta)-\mathrm{e}^{-t}=\left(\mathrm{e}^{-a(\rho)t}-\mathrm{e}^{-t}\right)+\frac{a(\rho)}{b(\rho)-a(\rho)}\left(\mathrm{e}^{-a(\rho)t}-\mathrm{e}^{-b(\rho)t}\right).
	\end{align*}
	The identity
	\begin{align*}
		\mathrm{e}^{-a(\rho)t}-\mathrm{e}^{-t}=-\bigl(a(\rho)-1\bigr)t\int_0^1\mathrm{e}^{-t[1+\vartheta(a(\rho)-1)]}\dd\vartheta
	\end{align*}
	shows that the first term is an exponentially decaying symbol of order $-2$, and the second term has the same property by the preceding estimates.
	
	After one time derivative, the slow-root terms remain of order $-2$, whereas the additional factor $b(\rho)$ in the fast-root terms is absorbed by $\mathrm{e}^{-b(\rho)t}$. For the first difference in the position multiplier, we also use
	\begin{align*}
		\mathrm{e}^{-t}-a(\rho)\mathrm{e}^{-a(\rho)t}=-\left(\mathrm{e}^{-a(\rho)t}-\mathrm{e}^{-t}\right)-\bigl(a(\rho)-1\bigr)\mathrm{e}^{-a(\rho)t}.
	\end{align*}
	Finally, the standard differentiation formulas for radial symbols and the compact support of the derivatives of $\chi_{\high}$ complete the proof.
\end{proof}

\begin{coro}[Spatial decay of the regular high-frequency kernels]\label{Coro-High-Spatial}
	For every $N\in\mb N$, there exists $C_N>0$ such that, for $t\geqslant1$ and $|x|\geqslant 1$,
	\begin{align*}
		|G_1^{(3),\high}(t,x)|+|\nabla G_1^{(3),\high}(t,x)|&\leqslant C_N\,\mathrm{e}^{-ct}(1+|x|)^{-N},\\
		|G_0^{(3),\high,\reg}(t,x)|&\leqslant C_N\,\mathrm{e}^{-ct}(1+|x|)^{-N}.
	\end{align*}
\end{coro}

\begin{proof}
	Apply sufficiently many powers of $-\Delta_\zeta$ to the multipliers in Lemma~\ref{Lemma-High-Symbols} and integrate by parts in the Fourier inversion formula for $|x|\geqslant 1$. The resulting differentiated symbols are integrable and hence give arbitrary spatial decay. For the position kernel, the contribution of the constant multiplier $\mathrm{e}^{-t}$ corresponds to the singular term $\mathrm{e}^{-t}\delta_0$ in physical space. After removing this contribution, the remaining high-frequency multiplier satisfies the symbol estimates in Lemma~\ref{Lemma-High-Symbols}, and the same argument applies.
\end{proof}

\section{Remainder estimates in the wave-front region}\label{Appendix-Front}

We complete the proof of Proposition~\ref{Prop-Front-Asymptotics}. Let $\chi_{\low}$ be the cutoff fixed in \cref{Section-Linear}. On its support, \eqref{Omega-Expansion} implies
\begin{align}\label{Omega-Error-Bounds}
	|\omega(\rho)-\rho|\lesssim\rho^3\ \ \text{and}\ \ \left|\frac1{\omega(\rho)}-\frac1\rho\right|\lesssim\rho.
\end{align}

To treat the velocity (second) kernel, let us define the sinc error
\begin{align*}
	E_1(t,\rho):=\mathrm{e}^{-\frac{t\rho^2}{2}}\left(\frac{\sin(t\omega(\rho))}{\omega(\rho)}-\frac{\sin(t\rho)}{\rho}\right).
\end{align*}
By using \eqref{Omega-Error-Bounds}, one easily gets
\begin{align}\label{E1-Bound}
	|E_1(t,\rho)|\lesssim\mathrm{e}^{-\frac{t\rho^2}{2}}(\rho+t\rho^2).
\end{align}
Let $r=t-\sigma\sqrt t$, where $\sigma$ belongs to a fixed compact subset of $(0,\infty)$. Then $r\approx t$, and the radial inversion formula as well as \eqref{E1-Bound} imply
\begin{align*}
	\left|\ml F^{-1}_{\rho\to r}[\chi_{\low}E_1](r)\right|\lesssim\frac1r\int_0^{2\rho_0}\mathrm{e}^{-\frac{t\rho^2}{2}}(\rho^2+t\rho^3)\dd\rho\lesssim t^{-2}.
\end{align*}
The contribution from the complementary frequency region is exponentially small. This proves the remainder estimate in \eqref{Front-G1}. Differentiating the radial inversion formula gives
\begin{align*}
	\partial_r\ml F^{-1}_{\rho\to r}[\chi_{\low}E_1](r)&=-\frac1{2\pi^2r^2}\int_0^\infty\chi_{\low}E_1(t,\rho)\rho\sin(r\rho)\dd\rho\\
	&\quad+\frac1{2\pi^2r}\int_0^\infty\chi_{\low}E_1(t,\rho)\rho^2\cos(r\rho)\dd\rho.
\end{align*}
The first term is $O(t^{-3})$. For the second term, \eqref{E1-Bound} yields
\begin{align*}
	\frac1t\int_0^{2\rho_0}\mathrm{e}^{-\frac{t\rho^2}{2}}(\rho^3+t\rho^4)\dd\rho
	&\lesssim\frac1t\left(t^{-2}+t^{1-\frac52}\right)\lesssim t^{-\frac52}.
\end{align*}
This demonstrates the remainder estimate in \eqref{Front-Dr-G1}.
\medskip

We consider now the position kernel. Set
\begin{align*}
	E_0(t,\rho):=\mathrm{e}^{-\frac{t\rho^2}{2}}\left(\cos\bigl(t\omega(\rho)\bigr)-\cos(t\rho)+\frac{\rho^2}{2\omega(\rho)}\sin\bigl(t\omega(\rho)\bigr)\right).
\end{align*}
By applying \eqref{Omega-Error-Bounds} again, one easily gets
\begin{align*}
	|E_0(t,\rho)|\lesssim\mathrm{e}^{-\frac{t\rho^2}{2}}(\rho+t\rho^3).
\end{align*}
Therefore, for $r=t-\sigma\sqrt t$ with $\sigma$ in a fixed compact subset of $(0,\infty)$, we may estimate
\begin{align*}
	\left|\ml F^{-1}_{\rho\to r}[\chi_{\low}E_0](r)\right|\lesssim t^{-1}\int_0^{2\rho_0}\mathrm{e}^{-\frac{t\rho^2}{2}}(\rho^2+t\rho^4)\dd\rho\lesssim t^{-\frac52}.
\end{align*}
The middle-frequency contribution is exponentially small by \eqref{Middle-Matrix}, while the high-frequency regular part is handled by Corollary~\ref{Coro-High-Spatial}. This completes the proof of Proposition~\ref{Prop-Front-Asymptotics}.

\section{Rapid decay in strict interior regions}\label{Appendix-Interior}

We prove Lemma~\ref{Lemma-Interior-Decay} here. The argument relies on the fact that the low-frequency phases have no stationary points in a strict interior region. We first establish the corresponding kernel estimates.

Fix $0<a<b<1$. Extending $\omega$ to an odd smooth function on $(-2\rho_0,2\rho_0)$ and decreasing $\rho_0$ if necessary, we may assume that
\begin{align}\label{Omega-Prime-Lower}
	\omega'(\rho)\geqslant\frac{1+b}{2}\ \ \text{for}\ \ |\rho|\leqslant2\rho_0.
\end{align}
By using the radial Fourier inversion formula, one expresses the kernel by
\begin{align*}
	rG_1^{(3),\low}(t,r)=\frac{1}{2\pi^2}\int_0^\infty\chi_{\low}(\rho)\mathrm{e}^{-\frac{t\rho^2}{2}}\frac{\rho}{\omega(\rho)}\sin\bigl(t\omega(\rho)\bigr)\sin(r\rho)\dd\rho.
\end{align*}
Since $\frac{\rho}{\omega(\rho)}$ extends smoothly across the origin as an even function, the function $a_1(\rho):=\chi_{\low}(\rho)\frac{\rho}{\omega(\rho)}$ satisfies $a_1\in\mathcal{C}_0^\infty\bigl((-2\rho_0,2\rho_0)\bigr)$ and $a_1(-\rho)=a_1(\rho)$. Hence, by the product-to-sum identity and parity,
\begin{align*}
	rG_1^{(3),\low}(t,r)=\frac{1}{8\pi^2}\bigl(I_-(t,r)-I_+(t,r)\bigr),
\end{align*}
where $I_\pm(t,r)$ are defined in \eqref{Interior-Oscillatory-General} with $a(\rho)=a_1(\rho)$.

The same argument applies to the position kernel, in other words,
\begin{align}\label{Interior-G0-Integral}
	rG_0^{(3),\low,\reg}(t,r)=\frac{1}{2\pi^2}&\int_0^\infty\chi_{\low}(\rho)\mathrm{e}^{-\frac{t\rho^2}{2}}\notag\\
	&\quad\quad\times\left[\rho\cos\bigl(t\omega(\rho)\bigr)+\frac{\rho^3}{2\omega(\rho)}\sin\bigl(t\omega(\rho)\bigr)\right]\sin(r\rho)\dd\rho.
\end{align}
Notice that $\chi_{\low}(\rho)\rho$ extends as an odd function in $\mathcal{C}_0^\infty\bigl((-2\rho_0,2\rho_0)\bigr)$, whereas $\chi_{\low}(\rho)\frac{\rho^3}{\omega(\rho)}$ extends as an even function in the same space. Hence, by the product-to-sum identities and parity, \eqref{Interior-G0-Integral} can also be written as a finite linear combination of oscillatory integrals of the form
\begin{align}\label{Interior-Oscillatory-General}
	I_\pm(t,r):=\int_{\mb R}\mathrm{e}^{-\frac{t\rho^2}{2}}a(\rho)\mathrm{e}^{i[t\omega(\rho)\pm r\rho]}\dd\rho.
\end{align}
It therefore remains to estimate \eqref{Interior-Oscillatory-General}.

Setting $y=\sqrt{t}\rho$, we obtain
\begin{align*}
	I_\pm(t,r)=t^{-\frac{1}{2}}\int_{\mb R}\mathrm{e}^{-\frac{y^2}{2}}a\left(\frac{y}{\sqrt{t}}\right)\mathrm{e}^{i\Psi_\pm(y)}\dd y
\end{align*}
with the phase function
\begin{align*}
\Psi_\pm(y):=t\omega\left(\frac{y}{\sqrt{t}}\right)\pm\frac{r}{\sqrt{t}}y.
\end{align*}
For $at\leqslant r\leqslant bt$, it follows from \eqref{Omega-Prime-Lower} that
\begin{align}\label{Nonstationary-Phase}
	|\Psi_-'(y)|&=\sqrt{t}\left|\omega'\left(\frac{y}{\sqrt{t}}\right)-\frac{r}{t}\right|\geqslant\frac{1-b}{2}\sqrt{t},\\
	|\Psi_+'(y)|&=\sqrt{t}\left[\omega'\left(\frac{y}{\sqrt{t}}\right)+\frac{r}{t}\right]\geqslant\left(\frac{1+b}{2}+a\right)\sqrt{t},
\end{align}
whenever $a\left(\frac{y}{\sqrt{t}}\right)\neq0$. Moreover, for any integer $m\geqslant2$, one arrives at
\begin{align}\label{Interior-Phase-Derivatives}
	\Psi_\pm^{(m)}(y)=t^{1-\frac{m}{2}}\omega^{(m)}\left(\frac{y}{\sqrt{t}}\right),
\end{align}
and consequently
\begin{align*}
	|\Psi_\pm^{(m)}(y)|\leqslant C_m t^{1-\frac{m}{2}}\ \ \text{for}\ \ a\left(\frac{y}{\sqrt{t}}\right)\neq0.
\end{align*}

Let us introduce
\begin{align*}
	\ml L_\pm:=\frac{1}{i\Psi_\pm'(y)}\partial_y\ \ \Rightarrow\ \ \ml L_\pm\mathrm{e}^{i\Psi_\pm(y)}=\mathrm{e}^{i\Psi_\pm(y)}.
\end{align*}
Its formal adjoint is given by
\begin{align*}
	\ml L_\pm^*f=-\partial_y\left(\frac{f}{i\Psi_\pm'(y)}\right).
\end{align*}
After integrating by parts $M\in\mb{N}$ times, we obtain
\begin{align}\label{Interior-IBP}
	I_\pm(t,r)=t^{-\frac{1}{2}}\int_{\mb R}\mathrm{e}^{i\Psi_\pm(y)}(\ml L_\pm^*)^M\left[\mathrm{e}^{-\frac{y^2}{2}}a\left(\frac{y}{\sqrt{t}}\right)\right]\dd y.
\end{align}
No boundary terms arise since the amplitude in \eqref{Interior-IBP} is smooth and compactly supported.

It follows from \eqref{Nonstationary-Phase} and \eqref{Interior-Phase-Derivatives} that, for every integer $j\geqslant0$, one computes
\begin{align*}
	\left|\partial_y^j\left(\frac{1}{\Psi_\pm'(y)}\right)\right|\leqslant C_jt^{-\frac{j+1}{2}}
\end{align*}
on the support of $a\left(\frac{y}{\sqrt{t}}\right)$. For all integers $j,K\geqslant0$, one has
\begin{align*}
	\left|\partial_y^j\left[\mathrm{e}^{-\frac{y^2}{2}}a\left(\frac{y}{\sqrt{t}}\right)\right]\right|\leqslant C_{j,K}\langle y\rangle^{-K}.
\end{align*}
Hence, by induction on $M$, for every integer $K\geqslant0$, we derive
\begin{align}\label{Interior-Adjoint-Bound}
	\left|(\ml L_\pm^*)^M\left[\mathrm{e}^{-\frac{y^2}{2}}a\left(\frac{y}{\sqrt{t}}\right)\right]\right|\leqslant C_{M,K}t^{-\frac{M}{2}}\langle y\rangle^{-K}.
\end{align}
Taking $K>1$ in \eqref{Interior-Adjoint-Bound}, we arrive at
\begin{align*}
	\sup_{at\leqslant r\leqslant bt}|I_\pm(t,r)|\leqslant C_Mt^{-\frac{M+1}{2}}.
\end{align*}
Since $M$ is arbitrary, the representations above imply that, for every $N>0$,
\begin{align}\label{Low-Interior-Rapid}
	\sup_{at\leqslant r\leqslant bt}r\left(|G_1^{(3),\low}(t,r)|+|G_0^{(3),\low,\reg}(t,r)|\right)\leqslant C_Nt^{-N}.
\end{align}

For the middle-frequency part, Lemma~\ref{Lemma-Middle-Multiplier} and the compactness of the frequency support imply exponential decay in time together with rapid decay in space. The regular high-frequency kernels satisfy the same bounds by Corollary~\ref{Coro-High-Spatial}. Moreover, the singular part of the position propagator is given by $\mathrm{e}^{-t}\delta_0$ and hence contributes $\mathrm{e}^{-t}u_0(x)$ after convolution. Since $u_0$ is compactly supported, this term vanishes in the region $|x|\approx t$ for all sufficiently large $t$.

It remains to pass from the kernel estimates to compactly supported initial data. Let
\begin{align*}
	\supp u_0\cup\supp u_1\subset B_R\ \ \text{with some}\  \  R>0.
\end{align*}
Provided that $at\leqslant|x|\leqslant bt$ and $|y|\leqslant R$, then, for all sufficiently large $t$,
\begin{align*}
	\frac{a}{2}t\leqslant|x-y|\leqslant\frac{1+b}{2}t\ \ \text{and}\ \ |x|\lesssim|x-y|.
\end{align*}
As a result, \eqref{Low-Interior-Rapid}, applied uniformly to the translated kernels $G_j(t,x-y)$, together with the compact support of the initial data, yields our desired estimate \eqref{Supp-07}.

\section*{Acknowledgments}
Wenhui Chen is supported in part by the National Natural Science Foundation of China (grant No. 12301270) and the Guangdong Basic and Applied Basic Research Foundation (grant No. 2025A1515010240).

\end{document}